\documentclass[12pt]{amsart}
\usepackage{amsmath,amssymb,txfonts}
\usepackage{amssymb}
\usepackage{amsxtra}
\usepackage{amsthm, color}
\usepackage{txfonts}
\usepackage{graphicx}
\usepackage{times}

\def\rr{{\mathbb R}}
\def\rn{{{\rr}^n}}

\def\nn{{\mathbb N}}
\def\zz{{\mathbb Z}}

\def\cs{{\mathcal S}}

\def\fz{\infty}
\def\az{\alpha}

\def\loc{{\mathop\mathrm{\,loc\,}}}

\def\BMO{{\mathop\mathrm{\,BMO\,}}}

\def\ez{\epsilon}

\def\kz{{\kappa}}

\def\ls{\lesssim}

\def\laz{\langle}
\def\raz{\rangle}

\def\r{\right}
\def\lf{\left}

\def\XXint#1#2#3{{\setbox0=\text{$#1{#2#3}{\int}$ }
\vcenter{\text{$#2#3$ }}\kern-.6\wd0}}

      \newtheorem{theorem}{Theorem}[section]

      \newtheorem{remark}[theorem]{Remark}
      
      \newtheorem{proposition}[theorem]{Proposition}

      \newcommand{\lt}[1]{[ {#1}] \lower.3ex\text{$_{t}$}}

\begin{document}

\title[Fractional Differential Couples by Sharp Inequalities and Duality Equations]
{Fractional Differential Couples\\  by Sharp Inequalities and Duality Equations}

\author{Liguang Liu}
\address{School of Mathematics,
Renmin University of China,
Beijing 100872, China}
\email{liuliguang@ruc.edu.cn}

\author{Jie Xiao}
\address{Department of Mathematics and Statistics,
Memorial University, St. John's, NL A1C 5S7, Canada}
\email{jxiao@math.mun.ca}

\thanks{
LL was supported by the National Natural Science Foundation of China
(\# 11771446);
JX was supported by NSERC of Canada (\# 202979463102000).
}


\date{\today}



\begin{abstract}
This paper presents a non-trivial two-fold study of the fractional differential couples - derivatives ($\nabla^{0<s<1}_+=(-\Delta)^\frac{s}{2}$) and gradients ($\nabla^{0<s<1}_-=\nabla (-\Delta)^\frac{s-1}{2}$) of basic importance in the theory of fractional advection-dispersion equations: one is to discover the sharp Hardy-Rellich ($sp<p<n$) $|$ Adams-Moser ($sp=n$) $|$ Morrey-Sobolev ($sp>n$) inequalities for $\nabla^{0<s<1}_\pm$; the other is to handle the distributional solutions $u$ of the duality equations $[\nabla^{0<s<1}_\pm]^\ast u=\mu$ (a nonnegative Radon measure) and $[\nabla^{0<s<1}_\pm]^\ast u=f$ (a Morrey function).
\end{abstract}

\maketitle

\tableofcontents

\arraycolsep=1pt
\numberwithin{equation}{section}

\section{Introduction}\label{s0}

In his celebrated 1988 paper \cite{Adams88}, Adams extends the Moser inequality in \cite{Mos} from the first order to the higher order gradients in the Euclidean space $\rr^{n\ge 2}$ - given the gradient $$\nabla=(\partial_{x_1},...,\partial_{x_n})$$ and the Laplacian
$$
\Delta=\sum_{j=1}^n\partial^2_{x_j}
$$
as well as
$$
\nabla^{m}=\begin{cases}(-1)^\frac{m}{2}(-\Delta)^{\frac m2}\ \ &\text{for}\ \ m\ \ \text{even}\\
(-1)^\frac{m-1}{2}\nabla(-\Delta)^{\frac {m-1}2}\ \ &\text{for}\ \ m\ \ \text{odd}
\end{cases}
\ \ \&\ \ 0<m<n,
$$
there is a constant {}{$c_{0,m,n}$} such that
\begin{equation}
\label{ada10}
\int_{\Omega}\exp\Bigg(\frac{\beta|u(x)|}{\|\nabla^mu\|_{L^\frac{n}{m}}}\Bigg)^\frac{n}{n-m}\,\frac{dx}{|\Omega|}\le c_{0,m,n}\ \ \forall\ \ u\in C^m_c(\Omega)
\end{equation}
holds, where:
\begin{itemize}
\item[$\rhd$]
$$
0\le\beta\le\beta_{0,m,n}=\begin{cases}
\lf(\frac n{\omega_{n-1}} \r)^{\frac {n-m}n} \frac{\pi^{\frac n2} 2^m \Gamma(\frac{m}{2})}{\Gamma(\frac{n-m}{2})}\ \ &\text{for}\ \ m \ \ \text{even}\\
\lf(\frac n{\omega_{n-1}} \r)^{\frac {n-m}n} \frac{\pi^{\frac n2} 2^m \Gamma(\frac{m+1}{2})}{\Gamma(\frac{n+1-m}{2})}\ \ &\text{for}\ \ m \ \ \text{odd}
\end{cases}
\ \ \&\ \ 0<m<n;
$$
\item[$\rhd$]
$\Omega$ is a subdomain of $\mathbb R^n$ with finite $n$-measure $|\Omega|$ and its associate space $C_c^m(\Omega)$ stands for all $C^m$-functions supported in $\Omega$;

\item[$\rhd$] $\Gamma(\cdot)$ is the standard gamma function and induces
${\omega_{n-1}=}\frac{2\pi^\frac{n}{2}}{\Gamma(\frac{n}{2})}$ - the area of the unit sphere $\mathbb S^{n-1}$ of $\rn$;

\item[$\rhd$] \eqref{ada10} is established through the Adams-Riesz potential inequality (just under \cite[(23)]{Adams88})
\begin{equation}
\label{aR}
|u(x)|\le\frac{\Big(\frac{n}{\omega_{n-1}}\Big)^\frac{n-m}{n}}{\beta_{0,m,n}}
\int_{\rn}{|y-x|^{m-n}|\nabla^m u(y)|}\,dy\ \ \ \ \forall\ \ u\in C_c^\infty.
\end{equation}
Moreover, if $\beta>\beta_{0,m,n}$ then there is $u\in C^m_c(\Omega)$ such that the integral in \eqref{ada10} can be made as large as desired - in other words - $\beta_{0,m,n}$ is sharp.
\end{itemize}

Upon examining $\|\nabla^mu\|_{L^\frac{n}{m}}$ in \eqref{ada10}, we are automatically suggested to consider a variant of \eqref{ada10} for
$$
\|\nabla^mu\|_{L^{1<p<\frac{n}{m}}}\ \ \text{or}\ \ \|\nabla^mu\|_{L^{\infty>p>\frac{n}{m}}}.
$$
\begin{itemize}
\item[$\rhd$] For the former, we use the $m$-form of \cite[Corollary 1 \& Theorem 4 (16)]{Be} to derive the sharp $m$-order Hardy-Rellich inequality
\begin{equation}
\label{eHR}\left(\int_{\rn}\bigg(\frac{|u(x)|}{|x|^m}\bigg)^p\,dx\right)^\frac1p\le c_{mp<n}\|\nabla^m u\|_{L^p}\ \ \forall\ \ u\in C^\infty_c,
\end{equation}
where
$$
c_{mp<n}=\begin{cases}
\frac{2^{-m}\Gamma\big(\frac{n}{2p}-\frac{m}{2}\big)\Gamma\big(\frac{n(p-1)}{2p}\big)}{\Gamma\big(\frac{n(p-1)}{2p}+\frac{m}{2}\big)\Gamma\big(\frac{n}{2p}\big)}\ \ &\text{for}\ m\ \text{even}\\
\Big(\frac{2^{1-m}p}{n-p}\Big)\left(\frac{\Gamma(\frac{n}{2p}-\frac{m}{2})\Gamma(\frac{n(p-1)}{2p}+\frac12)}{\Gamma(\frac{n(p-1)}{2p}+\frac{m}{2})\Gamma(\frac{n}{2p}-\frac12)}\right) \ \ &\text{for}\ m\ \text{odd}
\end{cases}\ \ \&\ \ 0<m<n.
$$
Of course, the case $m=1$ of \eqref{eHR} is the classical sharp Hardy inequality (cf. \cite{HLP}).

\item[$\rhd$] For the latter, we use the $m$-form of Theorem \ref{t21}(iii) (viewed as a sharp Morrey-Riesz inequality) and \eqref{aR} to discover the sharp $m$-order Morrey-Sobolev inequality
\begin{equation}
\label{eMS}
\|u\|_{L^\infty}\le \frac{\left(\frac{n(p-1)}{mp-n}\right)^\frac{p-1}{p}}{\beta_{0,m,n}}|\Omega|^\frac{mp-n}{pn}\|\nabla^m u\|_{L^p}\ \ \forall\ \ u\in C_c^m(\Omega).
\end{equation}
\end{itemize}
In particular, the case $m=1$ of \eqref{eMS} is the classical sharp Morrey-Sobolev inequality (cf. {\cite[Theorem~2.E.]{Ta}}).

Clearly, \eqref{ada10}, \eqref{eHR} and \eqref{eMS} give a complete structure on utilizing the higher derivatives and gradients to sharply dominate the size of a derivative/gradient-free function. However, upon recognizing the fractional vector calculus considerably used in both Herbst's study of the Klein-Gordon equation for a Coulomb potential \cite{Her} and Meerschaert-Mortensen-Wheatcraft's investigation of the particle mass density $u(x,t)$ of a contaminant in some fluid at a point $x\in\mathbb R^n$ at time $t>0$ which solves the fractional advection-dispersion equation (with a constant average velocity $\vec{v}$ of contaminant particles and a positive constant $\kappa$)
\begin{align*}
\partial_t u(x,t)&=-\vec{v}\cdot\nabla u(x,t)-\kappa (-\Delta)^\frac{1+s}{2} u(x,t)\\
&=-\vec{v}\cdot\nabla \rho(x,t)+\kappa\text{div}^s\big(\nabla u(x,t)\big)\\
&=-\vec{v}\cdot\nabla u(x,t)+\kappa \text{div}\big(\nabla^s_- u(x,t)\big)
\end{align*}
combining a fractional Fick's law for flux with a classic mass balance - and reversely- a fractional mass balance with a classic Fickian flux  \cite{MMW},
in the forthcoming sections we are driven to work out versions of \eqref{ada10}, \eqref{eHR} and \eqref{eMS} for the fractional differential couples - derivatives and gradients: $$\{\nabla^{0<s<1}_+,\nabla^{0<s<1}_-\}\ \ \text{corresponding naturally to}\ \ \{\nabla^{m=\text{even}}, \nabla^{m=\text{odd}}\},
$$
and their essential applications in the study of the distributional solutions to some fractional partial differential equations of dual character. More precisely,
\begin{itemize}
\item[$\rhd$] \S \ref{s1} collects some fundamental facts on $$\nabla^{0<s<1}_\pm\ \ \&\ \  [\nabla^{0<s<1}_\pm]^\ast$$ through the Stein-Weiss-Hardy inequalities and the  Fefferman-Stein type decompositions (cf. \cite{FS, BB, LXI}).

\item[$\rhd$] \S \ref{s2} utilizes Theorem \ref{t21} - an sharp embedding principle for the Riesz potentials to discover the fractional extensions of \eqref{ada10}, \eqref{eHR} and \eqref{eMS} - Theorem \ref{t31}.

\item[$\rhd$] \S\ref{s3} discusses the fractional Hardy-Sobolev spaces $$H^{0<s<1,1<p<\infty}\ \ \&\ \   H^{0<s<1,1<p<\infty}_\pm$$ and their dualities generated by $\nabla^{0<s<1}_\pm$ - Theorems \ref{t3.1}-\ref{t3.2}.
\item[$\rhd$] \S \ref{s4} studies the distributional solutions of the duality equations $$[\nabla^{0<s<1}_\pm]^\ast u=\mu
$$ for a nonnegative Radon measure $\mu$ and their absolutely continuous forms  $$[\nabla^{0<s<1}_\pm]^\ast u=
f
$$
under the hypothesis that $f$ is in the Morrey space
$\mathrm{L}^{1\le p<{\kappa}/{s},0<\kappa\le n}$ (cf. \cite{ADuke})
- Theorems \ref{t4.1}-\ref{t4.3}.
\end{itemize}

\noindent{\it Notation}. In what follows,
$U\lesssim V$ (resp.\, $U\gtrsim V$) means $U\le cV$ (resp. $U\ge cV$) for a positive constant $c$ and $U\approx V$ amounts to $U\gtrsim V\gtrsim U$.

\section{Fractional differential couples $\nabla^{0<s<1}_\pm$ and their dualities $[\nabla^{0<s<1}_\pm]^\ast$}\label{s1}

\subsection{Fractional differential couples $\nabla^{0<s<1}_\pm$}\label{s11}
For $(n,p)\in\mathbb N\times[1,\infty)$ let $H^p$ be the real Hardy space of all {}{functions $u$} in the Lebesgue space $L^p$ on the Euclidean space $\rn$ with
$$\|u\|_{H^p}=\|u\|_{L^p}+\|\vec{R}u\|_{L^p}<\infty,$$
where $\vec{R}={(R_1,\dots, R_n)}$  is the vector-valued Riesz transform on $\rn$, with
$$
\vec{R}u=(R_1u,\dots, R_nu)\ \ \&\ \
R_ju(x)=\left(\frac{\Gamma(\frac{n+1}{2})}{\pi^{\frac{n+1}{2}}}\right)\;\text{p.v.}\int_{\mathbb R^n}\frac{x_j-y_j}{|x-y|^{n+1}}u(y)\,dy
\ \ \textup{a.\,e.}\ x\in\rn.$$
Also, for a vector-valued function
$$\vec{f}=(f_1,\dots, f_n)$$
let
$$
\|{\vec{f}}\|_{L^p}={ \big\||\vec f|\big\|_{L^p}\approx }  \sum_{j=1}^n \|f_j\|_{L^p}.
$$
Note that
$H^p$ coincides with the classical Lebesgue space $L^p$ whenever $p\in(1,\infty)$ and the $(0,1)\ni s$-th order Riesz singular integral operator $I_s$ acting on a suitable function $u$ is defined by
$$
I_su(x)=\lf(\frac{\Gamma(\frac{n-s}{2})}{\pi^\frac{n}{2}2^s\Gamma(\frac{s}{2})}\r)\int_{\mathbb R^n}|x-y|^{s-n} u(y)\,dy
\qquad\textup{a.\,e.}\ x\in\rn.
$$
We refer the reader to Stein's seminal texts \cite{St, St2} for more about these basic notions.
The Stein-Weiss-Hardy inequality (cf. \cite{SW} for $p>1$ and \eqref{eSWH} in \S\ref{s4} for $p=1$)
states that under
$$
0<s<1\le p< \frac ns
$$
we have
\begin{equation}
\label{e1}
\left(\int_{\rn}\big(|x|^{-s}|I_s u(x)|\big)^p\,dx\right)^\frac1p\lesssim \|u\|_{L^p}+\|\vec{R}u\|_{L^p}{}{\approx\|u\|_{H^p}} \qquad \ \forall\ \ u\in  H^{p}.
\end{equation}

Let $C_c^\infty$ be the collection of all infinitely differentiable functions compactly supported in $\rn$.
Note that $C_c^\infty\cap H^p$ is dense in $H^p$ for any $p\in[1,\infty)$.
For any $u\in C_c^\infty$ let
\begin{align}\label{eq-def0}
(-\Delta)^\frac{s}{2}u(x)=\begin{cases}I_{-s}u(x)=c_{n,s}\int_{\mathbb R^n}\frac{u(x+y)}{|y|^{n+s}}\,dy\ \ &\text{as}\ \ s\in (-1,0)\\
u(x)  \ \ &\text{as}\ \ s=0\\
c_{n,s,+}\; \text{p.v.}\; \int_{\mathbb R^n}\frac{u(x+y)-u(x)}{|y|^{n+s}}\,dy\ \ &\text{as}\ \ s\in (0,1)
\end{cases}
\end{align}
and
\begin{align}\label{eq-def00}
\nabla^su(x)=\Bigg(\frac{\partial^s u}{\partial x_j^s}\Bigg)_{j=1}^n=\vec{R}(-\Delta)^\frac{s}{2}u(x)=c_{n,s,-}\int_{\mathbb R^n}\frac{y\big(u(x)-u(x-y)\big)}{|y|^{n+1+s}}\,dy,
\end{align}
where (cf. \cite[Definition~1.1, Lemma~1.4]{Bucur} for $c_{n,s,+}$ {and \cite{LXI} for $c_{n,s,-}$})
$$
\begin{cases}
c_{n,s}=\frac{\Gamma(\frac{n-s}{2})}{\pi^\frac{n}{2}2^s\Gamma(\frac{s}{2})}\\
c_{n,s,+}=\frac{s2^{s-1}\Gamma\big(\frac{n+s}{2}\big)}{\pi^{\frac{n}{2}}\Gamma\big(1-\frac{s}{2}\big)}\\
c_{n,s,-}=\frac{2^{s}\Gamma\big(\frac{n+s+1}{2}\big)}{\pi^{\frac n2}\Gamma\big(\frac{1-s}{2}\big)}.
\end{cases}
$$
Especially, if $0<s<n=1$ then there are two $s$-dependent constants $c_{\pm}$ to make the following Liouville fractional derivative formulae (cf. \cite{SS2}):
$$
\begin{cases}
(-\Delta)^\frac{s}{2}u(x)=c_+\bigg(\frac{d^s}{dx^s_+}+\frac{d^s}{dx^s_-}\bigg)u(x)\\
\nabla^su(x)=c_-\bigg(\frac{d^s}{dx^s_+}-\frac{d^s}{dx^s_-}\bigg)u(x)\\
\frac{d^s}{dx^s_\pm}u(x)=\frac{s}{\Gamma(1-s)}\int_{\pm\infty}^0\frac{t(u(x+t)-u(x))}{|t|^{2+s}}\,dt.
\end{cases}
$$
Hence it is natural and reasonable to adopt the notations
\begin{align}\label{eq-def}
\nabla^s_+u=(-\Delta)^\frac{s}{2}u\qquad \&\qquad \nabla^s_-u=\nabla^s u=\vec{R}(-\Delta)^\frac{s}{2}u.
\end{align}
The operators $\nabla^s_+$ and  $\nabla^s_-$ can be viewed as the fractional derivative and the fractional gradient due to
$$
\text{id}=-\sum_{j=1}^n R_j^2=-\vec{R}\cdot\vec{R}.
$$
Accordingly, for any $s\in(0,1)$,
the  Stein-Weiss-Hardy inequality \eqref{e1} (cf. \cite{SSS17}) amounts to
\begin{equation}
\label{e3}
\left(\int_{\rn}\big(|x|^{-s}|u(x)|\big)^p\,dx\right)^\frac1p
\lesssim \|\nabla^s_+u\|_{L^p}+\|\nabla^s_{-}u\|_{L^p} \qquad\forall\ \ u\in I_s(C_c^\infty\cap H^p).
\end{equation}
Here it is worth pointing out the following fundamentals:
\begin{itemize}

	\item[$\rhd$] If $0<s<1<p<\frac{n}{s}$, then the right-hand-side  of \eqref{e3} can be replaced by $\|\nabla^s_\pm u\|_{L^p}$.
	More precisely, on the one hand, the boundedness of $\vec R$ on $L^{p>1}$ and \eqref{e3} give (cf. \cite[Lemma 2.4]{SS1})
	\begin{equation*}
	\left(\int_{\rn}\big(|x|^{-s}|u(x)|\big)^p\,dx\right)^\frac1p\lesssim \|\nabla^s_{+}u\|_{L^{p}}\qquad\forall\ \  u\in I_s(C_c^\infty\cap H^p).
	\end{equation*}
	One the other hand, \cite[Theorems 1.8-1.9]{SS1} derives
	\begin{equation*}
	\left(\int_{\rn}\big(|x|^{-s}|u(x)|\big)^p\,dx\right)^\frac1p\lesssim \|\nabla^s_{-}u\|_{L^{p}}\qquad\forall\ \  u\in I_s(C_c^\infty\cap H^p).
	\end{equation*}
	
	\item[$\rhd$] If $0<s<p=1\le n$, then according to Spector's \cite[Theorem 1.4]{Sp} the right-hand-side of \eqref{e3} except $n=1$ (cf. \eqref{e1H}) can be replaced by $\|\nabla^s_{-}u\|_{L^{1}}$ - i.e. -
	\begin{equation*}
	\int_{\mathbb R^{n}}|x|^{-s}|u(x)|\,dx\lesssim \|\nabla^s_{-}u\|_{L^{1}}\ \ \text{under}\ \ n\ge 2\quad\forall\ \  u\in I_s(C_c^\infty\cap H^1).
	\end{equation*}
	which may be viewed as a rough extension of Shieh-Spector's \cite[Theorem 1.2]{SS2} and
	the classic sharp Hardy's inequality (cf. \cite{FrS}) under $n\ge 2$:
	\begin{equation*}
	\begin{cases}\int_{\rn}|x|^{-1}|u(x)|\,dx\le (n-1)^{-1}\|\nabla u\|_{L^1}&\qquad\forall\ \  u\in {C_c^\infty}\\
	\int_{\rn}|x|^{-1}|I_{1-s}u(x)|\,dx\le (n-1)^{-1}\|\nabla^s_- u\|_{L^1}&\qquad\forall\ \  u\in I_{1-s}(C_c^\infty).
	\end{cases}
	\end{equation*}
However, the right-hand-side of \eqref{e3} cannot be replaced by $\|\nabla^s_+u\|_{L^1}$ (cf. \cite[p.119]{St}, \cite[Section~3.3]{SSS17} \& {\cite[Section~1.1]{SS2})}.

\end{itemize}

\subsection{Dual fractional differential couples $[\nabla^{s}_\pm]^\ast$}\label{s12}
Suppose that $C^\infty$ is the space of all infinitely differentiable functions on $\rn$. Denote by $\cs$ the Schwartz class
on $\rn$ consisting of all functions $f$ in $C^\infty$ such that
$$
\rho_{N,\az}(f)=\sup_{x\in\rn}(1+|x|^N)|D^\az f(x)|<\infty\ \ \text{holds for}\ \
\begin{cases}
N\in{\zz_+}=\nn\cup\{0\}\\
\az=(\az_1,\dots,\az_n){\in\zz_+^n}\\
 D^\az=\partial_{x_1}^{\az_1}\cdots\partial_{x_n}^{\az_n}.
 \end{cases}
 $$
Also, write $\cs'$ for the Schwartz tempered distribution space - the dual of $\cs$ endowed with the weak-$\ast$ topology.
According to \cite{Sil, LXI}, given $s\in(0,1)$, if
we let
$$\cs_s=\lf\{f\in C^\infty:\  \rho_{n+s,\az}(f)=\sup_{x\in\rn}(1+|x|^{n+s})|D^\az f(x)|<\infty\;\ \forall\ \alpha\in\zz_+^n\r\}$$
and $\cs_s'$ be the dual space of $\cs_s$ (i.\,e., the space of all continuous linear functionals on $\cs_s$),
then for any
$$u\in\cs_s'\subseteq\cs'$$
we can define below $\nabla^s_\pm u$  as a distribution in $\cs'$:
\begin{align}\label{eq-def1}
\begin{cases}
\laz \nabla^s_+ u, \phi\raz =\laz u, \nabla^s_+\phi\raz \\
\nabla^s_-=(\nabla^s_1,\dots, \nabla^s_n)\\
\laz \nabla^s_j u, \phi\raz =-\laz u, \nabla^s_j\phi\raz \ \ \forall\ \ j\in\{1,...,n\}
\end{cases}
\ \  \forall \  \ \phi\in\cs,
\end{align}
{where the action of $\nabla^s_\pm$ on any function $\phi\in\cs$ is determined by the Fourier transform
$$\hat \phi(\xi)=\int_\rn \phi(x)e^{-2\pi i x\cdot\xi}\,dx \ \  \forall \  \  \xi\in\rn$$
according to
\begin{align}\label{eq-def2}
\begin{cases}
(\nabla^s_+\phi)^\wedge(\xi)=(2\pi|\xi|)^s \hat\phi(\xi)\\
(\nabla^s_j\phi)^\wedge(\xi)=(-2\pi i \xi_j)(2\pi|\xi|)^{s-1} \hat\phi(\xi)
\end{cases}
\ \  \forall \  \  \xi\in\rn.
\end{align}
If $\phi\in C_c^\infty$, then \eqref{eq-def2} goes back to  \eqref{eq-def0}-\eqref{eq-def00}-\eqref{eq-def} (cf. \cite{Sil, Bucur, LXI}). Moreover, the
above equalities in \eqref{eq-def1}} are well defined because  $\nabla^s_+$ and $\nabla^s_j$ send $\cs$ to $\cs_s$
(cf. \cite{Sil, Bucur} for $\nabla^s_+$
and \cite[Lemma~2.6]{LXI} for $\nabla^s_j$).

Based on the foregoing discussion, we may describe the dual/adjoint operators of $\nabla^s_{\pm}$ and one of their most important consequences.
\begin{itemize}
		\item[$\rhd$] The adjoint operator $\big[(-\Delta)^\frac{s}{2}\big]^\ast$ of $(-\Delta)^\frac{s}{2}$ is itself, namely, $$[\nabla^s_+]^\ast=(-\Delta)^\frac{s}{2},$$
	which can be understood in the sense of
\begin{align*}
	\laz [\nabla^s_+]^\ast f,\,\phi\raz=\laz f,\,\nabla^s_+ \phi \raz =\laz \nabla^s_+ f,\,\phi\raz
	\qquad \ \forall\, (f,\phi)\in\cs_s'\times\cs.
	\end{align*}
	This is reasonable, because for nice function pair $(f,\phi)\in (C_c^\infty)^2$ we have (cf. \cite{S})
$$\laz [\nabla^s_+]^\ast f,\,\phi\raz=\int_{\mathbb R^n}\big((-\Delta)^\frac{s}{2}f(x)\big)\phi(x)\,dx=\int_{\mathbb R^n}f(x)\big((-\Delta)^\frac{s}{2}\phi(x)\big)\,dx=\laz f,\,\nabla^s_+ \phi \raz$$
	and
	$$
	(-\Delta)^\frac{s}{2}\big((-\Delta)^\frac{s}{2}u\big)=(-\Delta)^{s}u\qquad \forall\ \ u\in C_c^\infty.
	$$
	
	\item[$\rhd$] Upon setting
	\begin{align*}
	\text{div}^s\vec{g}=(-\Delta)^{\frac s2}\vec{R}\cdot\vec{g},
	\end{align*}
then $-\text{div}^s$ exists as the adjoint operator $\big[\nabla^s_{-}\big]^\ast$  of $\nabla^s_{-}$ - in short - $$[\nabla^s_-]^\ast=-\text{div}^s.$$
Note that (cf. \cite[Theorem 1.3]{SS1})
	$$
	-\text{div}^s(\nabla^s_- u)=(-\Delta)^su \qquad \forall\ \  u\in C_c^\infty
	$$
	and (cf. \cite[Lemma 2.5]{CS})
	\begin{equation*}
	\int_{\mathbb R^n}f(x)(-\text{div}^s\vec{g})(x)\,dx=\int_{\mathbb R^n}\vec{g}(x)\cdot \nabla^s_- f(x)\,dx\qquad \forall\ \ (f,\vec{g})\in C_c^\infty\times (C_c^\infty)^n.
	\end{equation*}
	
\item[$\rhd$] Recall that $\mathrm{BMO}$ stands for the John-Nirenberg class of all locally integrable functions $f$ on $\rn$ with bounded mean oscillation (cf. \cite{JN})
$$
\|f\|_\mathrm{BMO}=\sup_{B\subseteq\rn} \frac1{{|B|}}\int_B |f(x)-f_B|\, dx<\infty
$$
where the supremum is taken over all Euclidean balls $B\subseteq\rn$ with $$
|B|=\int_{B}dx\ \ \&\ \
f_B=\frac1{{|B|}}\int_B f(x)\,dx.
$$
Of remarkable interest is that the
Fefferman-Stein decomposition (cf. \cite{FS, U})
$$
[H^1]^\ast=\BMO=
L^\infty+\vec{R}\cdot\big(L^\infty\big)^{n}
$$
can be written as the following form (cf. \cite[Theorem 4.4]{LXI})
 $$
 [H^1]^\ast=\BMO=
 L^\infty+I_s\big([\mathring{H}^{s,1}_{-}]^\ast\big),
 $$
 where
 $$
 \begin{cases}
 \mathring{H}^{s,1}_-
 =\text{closure of}\ \cs\ \text{in}\ H^{s,1}_-\ \text{under}\  [\cdot]_{H^{s,1}_-}\\
 H^{s,1}_\pm=\lf\{u\in\cs_s':\ [u]_{H^{s,1}_\pm}=\|\nabla^s_\pm u\|_{L^1}<\infty\r\}.
 \end{cases}
 $$
Note that if $W^{1,n}$ stands for the Sobolev space of all locally integrable functions $f$ with $\|\nabla f\|_{L^n}<\infty$ then there are (cf. \cite[Theorem 4.4]{LXI})
$$
W^{1,n}\subset I_s\big([\mathring{H}^{s,1}_{-}]^\ast\big)=\vec{R}\cdot\big(L^\infty\big)^{n}\subset\BMO\ \ \text{under}\ \ n\ge 2
$$
and (cf. \cite[theorem 1]{BB})
$$
W^{1,n}=\vec{R}\cdot\big(L^\infty\cap{W}^{1,n}\big)^n\ \ \text{under}\ \ n\ge 2.
$$
So, $I_s\big([\mathring{H}^{s,1}_{-}]^\ast\big)$ exists as a solution to the Bourgain-Brazis question (cf. \cite[p.396]{BB}) - {\it What are the function spaces $X$, $W^{1,n}\subseteq X\subseteq\BMO$, such that every $F\in X$ has a decomposition $F=\sum_{j=1}^n R_jY_j$ where $Y_j\in L^\infty$?}.
\end{itemize}

\section{Sharp fractional differential-integral inequalities}\label{s2}
\subsection{Optimal control for Riesz's operator $\mathrm{I}_{0<\alpha<n}$}\label{s2.1}

The following is of independent interest.

\begin{theorem}
	\label{t21}
	Let $$
	\begin{cases}
(p,\alpha)\in (1,\infty)\times(0,n)\\
	 I_\alpha=\left(\frac{\Gamma(\frac{n-\alpha}{2})}{2^\alpha\pi^\frac{n}{2}\Gamma(\frac{\alpha}{2})}\right)\mathrm{I}_\alpha{=c_{n,\alpha}\mathrm{I}_\alpha}\\
	\mathrm{I}_\alpha f=\int_{\rn}|x-y|^{\alpha-n}f(y)\,dy.
	\end{cases}
	$$
Then the following assertions are true.
	\begin{itemize}
		\item[\rm (i)] If $\alpha p<{n}$, then
		$$
		{\sup_{0\not=f\in L^p}}\frac{\left(\int_{\rn}\big(|x|^{-\alpha}|\mathrm{I}_\alpha f(x)|\big)^p\,dx\right)^\frac1p}{\|f\|_{L^p}}= c_{\az p<n}=\frac{2^{\frac{\alpha(p-1)}{p}}\pi^\frac{n}{2}\Gamma\big(\frac{\alpha}{2}\big)\Gamma\big(\frac{n}{2p}-\frac{\alpha}{2}\big)\Gamma\big(\frac{n(p-1)}{2p}\big)}{\Gamma\big(\frac{n-\alpha}{2}\big)\Gamma\Big(\frac{n(p-1)}{2p}+\frac{\alpha}{2}\Big)\Gamma\big(\frac{n}{2p}\big)}.
		$$

			\item[\rm (ii)] If $\alpha p=n$, $\Omega\subseteq\rn$ is a domain with volume $|\Omega|<\infty$ and $L^p_c(\Omega)$  stands for the class of all $f\in L^p$ with support contained in $\Omega$, then there is a constant
$c_{\az p=n}$ depending only on $\alpha$ and $n$ such that
			\begin{equation*}\label{ada}
			\sup_{f\in L_c^{p=\frac{n}{\alpha}}(\Omega)}\int_{\Omega}\exp\left(\beta\Bigg|\frac{\mathrm{I}_\alpha f(x)}{\|f\|_{L^{p=\frac{n}{\alpha}}}}\Bigg|^\frac{n}{n-\alpha}\right)\,
\frac{dx}{|\Omega|}\le  c_{\az p=n}\ \ \forall\ \ 0\le\beta\le\frac{n}{\omega_{n-1}}.
			\end{equation*}
		Here $\frac{n}{\omega_{n-1}}$ is sharp in the sense that if $\Omega$ is a Euclidean ball and $\beta>\frac{n}{\omega_{n-1}}$ then the last integral inequality cannot hold without forcing $c_{\alpha p=n}$ to depend only on $\alpha$ and $n$.
		
\item[\rm (iii)] If $\alpha p>{n}$ and $\Omega\subseteq\rn$ is a domain with volume $|\Omega|<\infty$, then
			$$
			\sup_{f\in L^p_c(\Omega)}\frac{\|\mathrm{I}_\alpha f\|_{L^\infty}}{\|f\|_{L^p}|\Omega|^{\frac{\alpha p-n}{pn}}}\le c_{\az p>n}=\left(\frac{\omega_{n-1}}{n}\right)^\frac{n-\alpha}{n}\left(\frac{n(p-1)}{\alpha p-n}\right)^{\frac{p-1}{p}}.
			$$
			Moreover, the constant $c_{\alpha p>n}$
			is sharp in the sense that if $\Omega$ is a Euclidean ball then
			$$
		\sup_{f\in L^p_c(\Omega)}\frac{\|\mathrm{I}_\alpha f\|_{L^\infty}}{|\Omega|^{\frac{\alpha p-n}{pn}} \|f\|_{L^p}}=c_{\alpha p>n}.
			$$
		\end{itemize}
	
\end{theorem}	
\begin{proof} (i) This is regarded as the sharp Stein-Weiss-Hardy inequality. The sharp constant $c_{\alpha p<n}$ is obtained in Herbst \cite{Her}; see also \cite{Be, Sa, GIT} for more information.

	(ii) This is just the sharp Adams inequality in \cite[Theorem 2]{Adams88} whose argument is still valid for $n=1$ and $\frac{\omega_{n-1}}{n}=2$.
	
	(iii) This is totally brand-new. In the sequel let
	$p'=\frac{p}{p-1}.
	$
	For any $f\in L^p$ supported on $\Omega$
		and for any
		$x\in \rn$,
		we utilize the H\"older inequality to derive that
\begin{align*}
	|\mathrm{I}_\alpha f(x)| \le\int_{\Omega}|f(y)||x-y|^{\alpha-n}\,dy
\le \|f\|_{L^p}\lf( \int_{\Omega} |x-y|^{(\alpha-n)p'}\,dy \r)^{\frac 1{p'}}.
	\end{align*}
Note that the Fubini theorem and $(\alpha-n)p'+n>0$ imply
\begin{align*}
\int_{\Omega} |x-y|^{(\alpha-n)p'}\,dy
&=(n-\alpha)p' \int_{\Omega} \lf(\int_{|x-y|}^\infty r^{(\alpha-n)p'-1}\, dr\r)\,dy \\
&=(n-\alpha)p' \int_0^\infty \lf(\int_{B(x,r)\cap \Omega} \,dy\r) r^{(\alpha-n)p'-1}\, dr \\
&\le (n-\alpha)p'  \int_0^\infty \min\lf\{\frac{\omega_{n-1}}{n}r^{n}, |\Omega|\r\}r^{(\alpha-n)p'-1}\, dr \\
&=(n-\alpha)p'  \lf(\frac{\omega_{n-1}}{n}
\int_0^{(\frac {n|\Omega|}{\omega_{n-1}})^\frac 1n} r^{(\alpha-n)p'+n-1}\,dr+ |\Omega|\int_{(\frac {n|\Omega|}{\omega_{n-1}})^\frac 1n}^\infty r^{(\alpha-n)p'-1}\, dr\r)\\
&= (n-\alpha)p'  \lf(\frac{1}{(\alpha-n)p'+n} +\frac 1{(n-\alpha)p'}\r)\left(\frac{\omega_{n-1}}{n}\right)^{\frac{(n-\az)p'}{n}}|\Omega|^{\frac{(\az-n)p'+n}{n}}\\
&=\left(\frac{n(p-1)}{\alpha p-n}\right)\left(\frac{\omega_{n-1}}{n}\right)^{\frac{(n-\az)p'}{n}}|\Omega|^{\frac{(\az-n)p'+n}{n}}.
	\end{align*}
Thus we arrive at the desired inequality
\begin{align*}
|\mathrm{I}_\alpha f(x)| \le  \|f\|_{L^p}\left(\frac{n(p-1)}{\alpha p-n}\right)^{\frac 1{p'}} \left(\frac{\omega_{n-1}}{n}\right)^\frac{n-\alpha}{n}{|\Omega|^{\frac{(\az-n)p'+n}{np'}}}.
	\end{align*}

To prove that
$$c_{\az p>n}= \left(\frac{\omega_{n-1}}{n}\right)^\frac{n-\alpha}{n}\left(\frac{n(p-1)}{\alpha p-n}\right)^{\frac{p-1}{p}}
$$
is sharp, let us consider the case
$$\Omega=B(x_0, r_0)\ \ \forall\ \ (x_0,r_0)\in\rn\times(0,\infty)$$ and the function
\begin{align*}
\rn\ni x\mapsto f_\beta(x)=1_{B(x_0, r_0)} |x-x_0|^\beta,
\end{align*}
where $\beta$ satisfies
$$\beta+\frac np>0.$$
On the one hand, a direct calculation gives
\begin{align*}
\|f_\beta\|_{L^p}
&= \lf(\int_{B(x_0, r_0)} |x-x_0|^{\beta p}\,dx\r)^{\frac 1p}\\
&=\lf(\omega_{n-1}\int_0^{r_0}r^{\beta p+n-1}\,dr\r)^{\frac 1p}\\
&=\lf(\frac{\omega_{n-1}}{\beta p+n}r_0^{\beta p+n}\r)^{\frac 1p}\\
&=\left(\frac{\omega_{n-1}}{n}\right)^{\frac1p} \lf(\frac n{\beta p+n}\r)^\frac 1 p r_0^{\beta+\frac np}.
\end{align*}
On the other hand, by the fact $\az+\beta>\az-\frac np>0$, we get
\begin{align*}
|\mathrm{I}_\alpha f_\beta(x_0)|
&=\int_{B(x_0, r_0)} |x-x_0|^{\az-n+\beta}\,dx
=\omega_{n-1}\int_0^{r_0}r^{\az+\beta p-1}\,dr
= \frac{\omega_{n-1}}{\az+\beta}r_0^{\az+\beta}.
\end{align*}
Combining the last two formulae gives
\begin{align*}
 c_{\az p>n}
 &\ge \sup_{x\in B(x_0,r_0)}\frac{|\mathrm{I}_\alpha f_\beta(x)|}{|B(x_0,r_0)|^{\frac{\alpha p-n}{np}} \|f_\beta\|_{L^p}} \\
 &\ge \frac{|\mathrm{I}_\alpha f_\beta(x_0)|}{\big(\frac{\omega_{n-1}}{n}\big)^{\frac{\alpha p-n}{np}}r_0^{\alpha-\frac{n}{p}} \|f_\beta\|_{L^p}}\\
 &= \left(\frac{\omega_{n-1}}{n}\right)^{\frac{n-\alpha }{n}} n^{1-1/p} \lf(\frac{\beta p+n}{(\az+\beta)^p}\r)^\frac 1p.
\end{align*}
Now the problem turns to calculate
$$
\sup_{\beta\in (-\frac np, \infty)} \frac{\beta p+n}{(\az+\beta)^p}.
$$
Consider the function
$$
-\frac{n}{p}<\beta\mapsto
h(\beta )= \frac{\beta  p+n}{(\az+\beta)^p}.
$$
Note that
\begin{align*}
h'(\beta)
= p(\az+\beta)^{-p}-p(\beta  p+n)(\az+\beta)^{-p-1}
=-p(\az+\beta)^{-p-1}(\beta(p-1)+n-\az).
\end{align*}
and
$$
\begin{cases}
h'(\beta)\ge 0\;\; \text{ if}\;\; \beta \le -\frac{n-\az}{p-1}\\
h'(\beta)\le 0\;\; \text{ if}\;\; \beta \ge -\frac{n-\az}{p-1}.
\end{cases}
$$
So, this, combined with
$$
\lim_{\beta\to -\frac np} h(\beta)=0,
$$
shows that $h$ attains its sharp value at the point
$$\beta=-\frac{n-\az}{p-1}.$$
Consequently,
$$
\sup_{\beta\in (-\frac np, \infty)} \frac{\beta p+n}{(\az+\beta)^p}
=\lf(\frac{\az p-n}{p-1}\r)^{1-p}.
$$
This in turn implies
\begin{align*}
 c_{\az p>n}
 &\ge \sup_{x\in B(x_0,r_0)}\frac{|\mathrm{I}_\alpha f_\beta(x)|}{|B(x_0,r_0)|^{\frac{\alpha p-n}{np}} \|f_\beta\|_{L^p}}\\
 &=\sup_{\beta\in (-\frac np, \infty)}\left(\frac{\omega_{n-1}}{n}\right)^{\frac{n-\alpha }{n}} n^{1-1/p} \lf(\frac{\beta p+n}{(\az+\beta)^p}\r)^\frac 1p\\
& =\left(\frac{\omega_{n-1}}{n}\right)^{\frac{n-\alpha }{n}}\left(\frac{n(p-1)}{\alpha p-n}\right)^{\frac{p-1}{p}}\\
&=c_{\alpha p>n}.
\end{align*}
Accordingly, {when $\Omega$ is a Euclidean ball of $\rn$, it holds that
$$
\sup_{f\in L^p_c(\Omega)}\frac{\|\mathrm{I}_\alpha f\|_{L^\infty}}{|\Omega|^{\frac{\alpha p-n}{pn}} \|f\|_{L^p}}=c_{\alpha p>n}.
$$}
\end{proof}

	\subsection{Optimal domination for $\nabla^{0<s<1}_\pm$}\label{s2.2}

Interestingly and naturally, with
$$\nabla^{m\in\{\text{even}\}}=(-1)^\frac{m}{2}(-\Delta)^{\frac m2}\qquad \text{or}\qquad  \nabla^{m\in\{\text{odd}\}}=(-1)^\frac{m-1}{2}\nabla(-\Delta)^{\frac {m-1}2}
=(-1)^\frac{m-1}{2}{\vec R}(-\Delta)^\frac{m}{2}
$$
replaced by the fractional version
$$\nabla^s_+=(-\Delta)^\frac s2\qquad \text{or}\qquad
\nabla^s_-=\nabla(-\Delta)^{\frac{s-1}{2}}{=\vec R(-\Delta)^\frac s2},$$
Theorem \ref{t21} induces the following new assertion.

\begin{theorem}\label{t31}
	Let $0<s<1<p<\infty$ and
$$
\mathcal{F}_{s,\pm}(\Omega)=\begin{cases}
I_s\big(C^\infty_c(\Omega)\big)\ \ &\text{for}\ \ \nabla^s_+\\
(-\Delta)^\frac{1-s}{2} \big(C_c^\infty(\Omega)\big) \ \ &\text{for}\ \ \nabla^s_-.
\end{cases}
$$
Then the following assertions are true.	
\begin{itemize}
		
\item[\rm (i)] If $sp<p<n$, then
$$
\sup_{g\in{C_c^\infty}} \frac{\left(\int_{\rn}\big(|x|^{-s}|g(x)|\big)^p\,dx\right)^\frac1p}{\|\nabla^s_\pm g\|_{L^p}}=
		\kappa_{sp<n,\pm}=\begin{cases}
		 \frac{2^{-\frac{s}{p}}\Gamma(\frac{n}{2p}-\frac{s}{2})\Gamma(\frac{n(p-1)}{2p})}{\Gamma(\frac{n(p-1)}{2p}+\frac{s}{2})\Gamma(\frac{n}{2p})}
		\ \ &\text{for}\ \ \nabla^s_+\\
		 \Big(\frac{2^{1-s}p}{n-p}\Big)\left(\frac{\Gamma(\frac{n}{2p}-\frac{s}{2})\Gamma(\frac{n(p-1)}{2p}+\frac12)}{\Gamma(\frac{n(p-1)}{2p}+\frac{s}{2})\Gamma(\frac{n}{2p}-\frac12)}\right) \ \ &\text{for}\ \ \nabla^s_-.
		\end{cases}
		$$
		
\item[\rm (ii)] If $sp=n$ and $\Omega\subseteq\rn$ is a domain with volume $|\Omega|<\infty$,
then exists a positive constant $c_{sp=n,\pm}$ depending only on $s$ and $n$ such that
	$$
	\sup_{g\in{\mathcal{F}_{s,\pm}(\Omega)}}\int_\Omega\exp\left(\frac{\kappa |g(x)|}{\|\nabla^s_\pm g\|_{L^\frac{n}{s}}}\right)^\frac{n}{n-s}\,\frac{dx}{|\Omega|}\le c_{sp=n,\pm}\ \ \ \forall\ \ 0\le\kappa\le{\kappa_{sp=n,\pm}}.
	$$
	Here
$$
\kz_{sp=n,\pm}=\begin{cases}
	\lf(\frac n{\omega_{n-1}} \r)^{\frac {n-s}n} \frac{\pi^{\frac n2} 2^s \Gamma(\frac{s}{2})}{\Gamma(\frac{n-s}{2})}\ \ &\text{for}\ \ \nabla^s_+\\
	\lf(\frac n{\omega_{n-1}} \r)^{\frac {n-s}n} \frac{\pi^{\frac n2} 2^s \Gamma(\frac{s+1}{2})}{\Gamma(\frac{n+1-s}{2})}\ \ &\text{for}\ \ \nabla^s_-
	\end{cases}
	$$
	is sharp in the sense that if $\Omega$ is a Euclidean ball and $\kappa>\kappa_{sp=n,\pm}$ then the last integral inequality cannot hold without forcing $c_{sp=n,\pm}$ to depend only on $s$ and $p$.

	\item[\rm (iii)] If $sp>n$ and $\Omega\subseteq\rn$ is a domain with volume $|\Omega|<\infty$
then
	$$
	\sup_{g\in{\mathcal{F}_{s,\pm}(\Omega)}}\frac{|\Omega|^\frac{s p-n}{pn}\|g\|_{L^\infty}}{\|\nabla^s_\pm g\|_{L^p}}\le \kappa_{sp>n,\pm}=\begin{cases}
c_{sp>n}\left(\frac{\Gamma(\frac{n-s}{2})}{2^s\pi^\frac{n}{2}\Gamma(\frac{s}{2})}\right)\ \ &\text{for}\ \ \nabla^s_+\\
c_{sp>n}\left(\frac{\Gamma(\frac{n-s+1}{2})}{2^s\pi^\frac{n}{2}\Gamma(\frac{1+s}{2})}\right) \ \ &\text{for}\ \ \nabla^s_-.
\end{cases}
$$
Moreover, the constant $\kappa_{sp>n,\pm}$ is sharp in the sense that if $\Omega$ is a Euclidean ball then
$$
\sup_{g\in{\mathcal{F}_{s,\pm}(\Omega)}}\frac{\|g\|_{L^\infty}}{\,|\Omega|^\frac{s p-n}{pn}\|\nabla^s_{\pm}g\|_{L^p}}=\kappa_{sp>n,\pm}.
$$
\end{itemize}

\end{theorem}

\begin{proof} The sharp inequalities in (i), (ii) and (iii) are suitably called the sharp Hardy-Rellich, Adams-Moser and Morrey-Sobolev inequalities for the fractional order twin gradients $\nabla^{0<s<1}_\pm$, respectively. Since (i) follows readily from \cite[Corollary 1 \& Theorem 4 (16)]{Be}, the definition of $\nabla^s_\pm$ and $I_s=(-\Delta)^{-\frac{s}{2}}$, it remains to verify (ii)-(iii).

{\it Case - $\nabla^s_+$}. Under this situation we have
$$
	g\in I_s\big(C_c^\infty(\Omega)\big)\Longleftrightarrow \exists\ u\in C_c^\infty(\Omega)\ \text{such that}\ g=I_su
	$$
and
	$$
\nabla^s_+ g=(-\Delta)^\frac{s}{2}I_su=u\in C_c^\infty(\Omega).
	$$
This, along with {Theorem \ref{t21}}(ii)/(iii), directly gives the desired conclusion in  (ii)/(iii) for $\nabla^s_+$ and the corresponding sharp case.

{\it Case - $\nabla^s_-$}. From the hypothesis $$g\in (-\Delta)^\frac{1-s}{2} \big(C_c^\infty\big)
	$$
	it follows that
	$$g=(-\Delta)^\frac{1-s}{2} u\quad\text{for some}\quad u\in C_c^\infty,$$
	and hence
	$$
	\nabla^s_-g=\nabla^s_- (-\Delta)^\frac{1-s}{2} u=\nabla u.
	$$
	Also, according to \cite[(5.6)\&(4.4)]{S} we have
\begin{align}\label{eq-frac1}
	\begin{cases}
	-(-\Delta)^\frac{1-s}{2}u=\text{div}^{-s}\nabla u=\kappa_{-s}\int_{\mathbb R^n}\frac{h\cdot\nabla u(x+h)}{|h|^{n-s+1}}\,dh\\
	\kappa_{-s}=\frac{\Gamma\big(\frac{n-s+1}{2}\big)}{2^{s}\pi^{\frac{n}{2}}\Gamma\big(\frac{1+s}{2}\big)},
	\end{cases}
\end{align}
	thereby finding
	\begin{align}\label{eq-frac}
	|g(x)|=|(-\Delta)^\frac{1-s}{2}u(x)|\le\kappa_{-s}\int_{\rn}{|x-y|^{s-n}}{|\nabla u(y)|}\,dy=\kappa_{-s}\mathrm{I}_s|\nabla u|(x),
	\end{align}
which exists as a fractional variant of \eqref{aR}. In light of \eqref{eq-frac} and Theorem \ref{t21}(ii)/(iii), we obtain the desired inequality in Theorem \ref{t31}(ii)/(iii).

To see that $\kappa_{sp\ge n,-}$ is sharp, we consider two situations below.

\begin{itemize}

\item[$\rhd$] $sp=n$. Without loss of generality we may assume that $\Omega$ is the origin-centered unit ball $\mathbb B^n$. If for some $\kappa>\kz_{sp=n,-}$ it holds that
	\begin{equation}
	\label{eB}
	\sup_{u\in C^\infty_c(\mathbb B^n)}\int_{\mathbb B^n}\exp\left(\frac{\kappa|(-\Delta)^\frac{1-s}{2}u(x)|}{\big\||\nabla u|\big\|_{L^p}}\right)^\frac{n}{n-s}\,\frac{dx}{|\mathbb B^n|}=\sup_{g\in\mathcal{F}_{s,-}(\mathbb B^n)}\int_{\mathbb B^n}\exp\left(\frac{\kappa| g(x)|}{\|\nabla^s_- g\|_{L^p}}\right)^\frac{n}{n-s}\,\frac{dx}{|\mathbb B^n|}\le c_{sp=n,-},
	\end{equation}
then we are about to construct suitable functions $u$ to show that \eqref{eB} forces
$\kappa\le\kz_{sp=n,-}$, thereby revealing that $\kz_{sp=n,-}$ is the sharp number to guarantee Theorem \ref{t31}(ii).

{Being somewhat motivated by \cite[pp.391-392]{Adams88} and  \cite[p.7]{Fu}, for $r\in (0,1)$} we let $\mathbb B^n_r$ be the origin-centered ball with radius $r$ and
$$
u_r(x)=\frac{|x|^{1-s}1_{\mathbb B^n\setminus\mathbb B^n_r}(x)}{(1-s)\omega_{n-1}\log\frac1r}.
$$
Then
	\begin{equation}
	\label{eB1}
	\begin{cases}
	\nabla u_r(x)=\frac{x|x|^{-1-s} 1_{\mathbb B^n\setminus\mathbb B^n_r}(x)}{\omega_{n-1}\log\frac1r}\\
	\big\||\nabla u_r|\big\|_{L^p}=\Big(\omega_{n-1}\log\frac1r\Big)^\frac{1-p}{p}=\Big(\omega_{n-1}\log\frac1r\Big)^\frac{s-n}{n}.
	\end{cases}
	\end{equation}
	Consequently, we use the first equation in \eqref{eB1}, \eqref{eq-frac1} and the polar-coordinate-system to achieve that if $x\in\mathbb B^n_r$ then
	
	\begin{align*}
	{-(-\Delta)^\frac{1-s}{2}u_r(x)}&=\kappa_{-s}\int_{\mathbb R^n}\frac{h\cdot\nabla u_r(x+h)}{|h|^{n-s+1}}\,dh\\
	&=\kappa_{-s}\int_{\mathbb B^n\setminus\mathbb B^n_r}\frac{(z-x)\cdot\nabla u_r(z)}{|z-x|^{n+1-s}}\,dz\\
	&=\left(\frac{\kappa_{-s}}{\omega_{n-1}\log\frac1r}\right)\int_{\mathbb B^n\setminus\mathbb B^n_r}\frac{(z-x)\cdot \frac{z}{|z|^{1+s}}}{|z-x|^{n+1-s}}\,dz\\
		&=\left(\frac{\kappa_{-s}}{{\omega_{n-1}}\log\frac1r}\right)\int_r^1\left(\int_{\mathbb S^{n-1}}\frac{\big(\theta-\frac{x}{\rho}\big)\cdot \theta}{\big|\theta-\frac{x}{\rho}\big|^{n+1-s}}\,d\theta\right)\frac{d\rho}{\rho}\\
		 &=\left(\frac{\kappa_{-s}}{\log\frac1r}\right)\int_{|x|}^{\frac{|x|}{r}}\left({\frac1{\omega_{n-1}}}\int_{\mathbb S^{n-1}}\frac{\big(\theta-t\frac{x}{|x|}\big)\cdot \theta}{\big|\theta-t\frac{x}{|x|}\big|^{n+1-s}}\,d\theta\right)\frac{dt}{t}\\
		&=\left(\frac{\kappa_{-s}}{\log\frac1r}\right)\int_{|x|}^{\frac{|x|}{r}} U(t)\,\frac{dt}{t},
		\end{align*}
where
$$U(t)=\frac1{\omega_{n-1}}\int_{\mathbb S^{n-1}}\frac{\big(\theta-t\frac{x}{|x|}\big)\cdot \theta}{\big|\theta-t\frac{x}{|x|}\big|^{n+1-s}}\,d\theta$$
{is independent of the variable $x$ after a rotation.}
Since $U(0)=1$, we write
\begin{align*}
	\int_{|x|}^{\frac{|x|}{r}} U(t)\,\frac{dt}{t}&=\int_{|x|}^{\frac{|x|}{r}}U(0)\,\frac{dt}{t}+\int_{|x|}^{\frac{|x|}{r}} \big(U(t)-U(0)\big)\,\frac{dt}{t}=\log\frac1r+T(|x|,r).
	\end{align*}
For the error term $T(|x|,r)$, observing that
\begin{align*}
\int_0^1 |U(t)-U(0)|\, \frac{dt}{t}
&\le \int_0^{1/2}  t\sup_{\tau\in(0,1/2)}|\nabla U(\tau)|\, \frac{dt}{t}+\int_{1/2}^1 |U(t)-1|\, \frac{dt}{t}\ls 1,
\end{align*}
we therefore derive that,  for $\epsilon>0$ there is a sufficiently small $r_0>0$ such that
	$$
	\sup_{x\in \mathbb B_r^n}\left|T(|x|,r)\Big(\log\frac1r\Big)^{-1}\right|\le \sup_{x\in \mathbb B_r^n}\left|\frac1{\log\frac1r}\int_0^1 |U(t)-U(0)|\, \frac{dt}{t}\right| <\epsilon\ \ \ \forall\ \ 0<r\le r_0.
	$$
So, we have
\begin{align*}
|(-\Delta)^\frac{1-s}{2}u_r(x)|
\ge \kz_{-s}(1-\ez) \qquad \forall\ \ (x,r)\in \mathbb B_r^n\times(0, r_0].
\end{align*}
This, along with \eqref{eB} and the second formula of \eqref{eB1}, gives
	\begin{align*}
	c_{sp=n,-}
\ge \int_{\mathbb B_r^n}\exp\left(\frac{\kappa|(-\Delta)^\frac{1-s}{2}u_r(x)|}{\big\||\nabla u_r|\big\|_{L^p}}\right)^\frac{n}{n-s}\,\frac{dx}{|\mathbb B^n|}
\ge r^n\exp\left(\frac{\kz\kz_{-s}(1-\ez)}{\Big(\omega_{n-1}\log\frac1r\Big)^\frac{s-n}{n}}\right)^\frac{n}{n-s},
	\end{align*}
	which in turns implies that if $0<r\le r_0$ then
	$$
\kz\kz_{-s}(1-\ez)\le \lf(\log\frac{c_{sp=n,-}}{r^n}\r)^\frac{n-s}{n} \Big(\omega_{n-1}\log\frac1r\Big)^\frac{s-n}{n}
=\lf(\frac{\log\frac{c_{sp=n,-}}{r^n}}{\omega_{n-1}\log\frac1r}\r)^\frac{n-s}{n}.$$
Letting {$\ez\downarrow0$ and $r\downarrow0$} yields
	$$
	\kappa\kappa_{-s}\le\left(\frac{n}{\omega_{n-1}}\right)^\frac{n-s}{n}\ \ \text{i.e.}\ \ \ \kappa\le\kappa_{sp=n,-}=\big(\kappa_{-s}\big)^{-1}\left(\frac{n}{\omega_{n-1}}\right)^\frac{n-s}{n},
	$$
	as desired.

\item[$\rhd$] $sp>n$. Let
$$
\begin{cases}(x_0,r_0)\in\rn\times(0,\infty)\\
\Omega=B(x_0, r_0)\\
\beta=-\frac{n-s}{p-1}\\
u_\beta(x)=(\beta+1)^{-1} 1_{B(x_0, r_0)} |x-x_0|^{\beta+1}\\
g_\beta(x)= (-\Delta)^\frac{1-s}{2} u_\beta(x).
\end{cases}
$$
Notice that $u_\beta$ can be approximated by functions in $C_c^\infty$ and
%
$$\nabla^s_-g_\beta(x)=\nabla u_\beta (x)= 1_{B(x_0, r_0)} |x-x_0|^{\beta} \frac{x-x_0}{|x-x_0|}.$$
So, by \eqref{eq-frac1} and the calculations in the proof of Theorem \ref{t21}(iii), we obtain
$$
\|\nabla^s_-g_\beta\|_{L^p}=\lf(\int_{B(x_0, r_0)}|x-x_0|^{\beta p}\, dx\r)^{\frac 1p}
= \left(\frac{\omega_{n-1}}{n}\right)^{\frac1p} \lf(\frac n{\beta p+n}\r)^\frac 1 p r_0^{\beta+\frac np}
$$
and
\begin{align*}
|g_\beta(x_0)|
= \kappa_{-s}\lf|\int_{\mathbb R^n}\frac{h\cdot\nabla u_\beta(x_0+h)}{|h|^{n-s+1}}\,dh\r|
=  \kappa_{-s}\int_{|h|<r_0} |h|^{\beta+s-n}\, dh
=  \kappa_{-s} \left(\frac{\omega_{n-1}}{\beta+s}\right) r_0^{\beta+s}.
\end{align*}
This in turn implies
\begin{align*}
\kappa_{sp>n,-}&\ge\sup_{g\in\mathcal{F}_{s,-}(B(x_0,r_0))}\frac{\|g\|_{L^\infty(B(x_0,r_0))}}{\,|B(x_0,r_0)|^\frac{s p-n}{pn}\|\nabla^s_{\pm}g\|_{L^p}}\\
&\ge\frac{|g_\beta(x_0)|}{\,|B(x_0,r_0)|^\frac{s p-n}{pn}\|\nabla^s_{\pm}g_\beta\|_{L^p}}\\
&=\kappa_{-s}c_{sp>n}\\
&=\kappa_{sp>n,-},
\end{align*}
and so $\kappa_{sp>n,-}$ is sharp.
\end{itemize}
\end{proof}

\section{Fractional Hardy-Sobolev spaces and their dualities}\label{s3}

\subsection{Fractional Hardy-Sobolev spaces $H^{s,p}$ and $H^{s,p}_\pm$}\label{s3.1}

Suppose $0<s<1\le p<\infty.
$
Since both $\nabla^s_+ u$ and $\nabla^s_-u$ are well defined when $u\in\cs_s'$,
the study for the case $p=1$ of \eqref{e3} in \cite{LXI} motivates us to consider the
fractional Hardy-Sobolev space
$$
H^{s,p}=\lf\{u\in\cs_s':\ [u]_{H^{s,p}}=\|(-\Delta)^\frac s2 u\|_{H^p}<\infty\r\}.
$$
Note that
$$u_1-u_2=\text{constant}\ \Longleftrightarrow\ [u_1]_{H^{s,p}}=[u_2]_{H^{s,p}}.
$$
So, $[\cdot]_{H^{s,p}}$ is properly a norm on quotient space of $H^{s,p}$ modulo the space of all real constants,
and consequently this quotient space is {a Banach space.}

Upon introducing
$$
H^{s,p}_\pm=\lf\{u\in\cs_s':\ [u]_{H^{s,p}_\pm}=\|\nabla^s_\pm u\|_{L^p}<\infty\r\},
$$
we find immediately
$$
H^{s,p}=H^{s,p}_+\cap H^{s,p}_{-}.
$$
Indeed, as shown in the next theorem, when $s\in(0,1)$ and $p\in(1,\infty)$, these three spaces are equal to each other and they
all have {the Schwartz class} $\cs$ and
$$\cs_\infty=\{\phi\in\cs:\, \text{the Fourier transform of $\phi$ is $0$ near the origin}\}$$
as dense subspaces.

\begin{theorem}\label{t3.1}
Let $0<s<1<p<\infty$. Then
$$
\cs_\infty{\subseteq} \cs\subseteq H^{s,p}=H^{s,p}_+=H^{s,p}_-. 
$$
Moreover, both $\cs_\infty$ and $\cs$ are dense in $H^{s,p}$ and $H^{s,p}_\pm$.
\end{theorem}

\begin{proof}
Notice that any $u\in\cs$ satisfies $(-\Delta)^\frac s2 u\in\cs_s$ (cf. \cite{Sil}). Of course, any function in $\cs_s$ belongs to $L^{1<p<\infty}$.
We therefore obtain
$$\cs\subseteq {H^{s,p}}.$$

{Given $p\in(1,\infty)$, upon} recalling boundedness of the Riesz transforms $R_j$ on $L^p$ {(cf. \cite{St})} and the identity
$$\text{id}=-\sum_{j=1}^n R_j^2 \quad\text{in}\;\; L^p,$$
we achieve
$$\|f\|_{L^p}+\|\vec R f\|_{L^p}\approx \|f\|_{L^p}\approx \|\vec Rf\|_{L^p} \qquad \forall\; f\in L^p,$$
thereby reaching
$$
 \|(-\Delta)^\frac s2 u\|_{L^p}+\|\vec R (-\Delta)^\frac s2 u\|_{L^p}
\approx \|(-\Delta)^\frac s2 u\|_{L^p}
\approx\|\vec R (-\Delta)^\frac s2 u\|_{L^p}.
$$
This in turn implies
$$
	[u]_{H^{p,s}}\approx [u]_{H^{s,p}_+}\approx [u]_{H^{s,p}_-}.
$$
Consequently, we obtain
\begin{equation*}
H^{s,p}=H^{s,p}_+=H^{s,p}_-.
\end{equation*}

It suffices to show the density of $\cs_\infty$ in $H^{s,p}_+$. If $u\in H^{s,p}_+$, then
$$u\in\cs_s'\ \ \&\ \ (-\Delta)^\frac s2 u\in L^p.
$$
Due to the density of $\cs_\infty$ in $L^p$ (cf. the proof of \cite[Lemma~2.9(iii)]{LXI}), we can find a sequence $\{f_j\}_{j\in\nn}$ in $\cs_\infty$ such that
$$\lim_{j\to\infty}\|f_j- (-\Delta)^\frac s2 u\|_{L^p}= 0.$$
For any $j\in\nn$, we write
$$u_j=I_sf_j\in\cs_\infty.
$$
Upon noticing
$$f_j=(-\Delta)^\frac s2u_j,
$$
we obtain
$$
[u_j-u]_{H^{s,p}_+}=\|(-\Delta)^\frac s2 (u_j- u)\|_{L^p}=\|f_j- (-\Delta)^\frac s2 u\|_{L^p}\to0\qquad\text{as}\ \ j\to\infty.
$$
Thus, any $u\in H^{s,p}_+$  can be approximated by the $\cs_\infty$-functions $\{u_j\}_{j\in\nn}$.
\end{proof}

\subsection{Dual Hardy-Sobolev spaces $[H^{s,p}]^\ast$ and $[H^{s,p}_\pm]^\ast$}\label{s3.2}
In this subsection, we are about to show that these dual spaces can be characterized by
$$(T_0, T_1,\dots, T_n)\in\big( L^{\frac{p}{p-1}}\big)^{n+1}$$
solving the fractional differential equation
$$[\nabla^s_+]^\ast T_0=T\ \ \ \text{or}\ \ \ [\nabla^s_-]^\ast(T_1,\dots, T_n)=T.$$

\begin{theorem}\label{t3.2}
Let	{$0<s<1<p<\infty$}  and  $p'=\frac{p}{p-1}.$ Then for any distribution $T\in\cs'$ the following three assertions are equivalent:
\begin{itemize}
\item[\rm (i)] $T\in [H^{s,p}]^\ast=[H^{s,p}_+]^\ast=[H^{s,p}_-]^\ast$;
\item[\rm (ii)] $\exists$\ $T_0\in L^{p'}$ such that $T=[\nabla_+^s]^\ast T_0$ in $\cs'$;
\item[\rm (iii)] $\exists$\ $(T_1,\dots, T_n)\in (L^{p'})^{n}$ such that $T= [\nabla_-^s]^\ast (T_1,\dots, T_n)$ in $\cs'$.
\end{itemize}
\end{theorem}

\begin{proof}

Note that {Theorem \ref{t3.1}} implies
\begin{equation}
\label{e32}
[H^{s,p}]^\ast=[H^{s,p}_+]^\ast=[H^{s,p}_-]^\ast.
\end{equation}
So, we begin with showing that (ii) implies (i)  by considering $H^{s,p}_+$.
If (ii) is valid, i.e., if
	$$T=[\nabla^s_+]^\ast T_0\ \ \text{in}\ \ \cs' \ \ \text{for some}\ \  T_0\in L^{p'},
	$$
	then
	\begin{align*}
	\laz T, \phi\raz
=\laz [\nabla^s_+]^\ast T_0, \phi\raz
	=\laz  T_0, \nabla^s_+\phi\raz
	=\laz  T_0, (-\Delta)^{\frac s2}\phi\raz\qquad \ \forall\ \ \phi\in\cs,
	\end{align*}
	and hence
	\begin{align*}
	|\laz T, \phi\raz| \le \|T_0\|_{L^{p'}} \| (-\Delta)^{\frac s2}\phi\|_{L^p}=\|T_0\|_{L^{p'}}[\phi]_{H^{s,p}_+} \qquad \ \forall\ \ \phi\in\cs.
	\end{align*}
	Accordingly, using the density of $\cs$ in $H^{s,p}_+$, we see that $T$ induces a bounded linear functional on $H^{s,p}_+$.
This proves that
$$T\in [H^{s,p}_+]^\ast$$ and (i) holds due to \eqref{e32}.
	
	Conversely,	in order to show that (i) implies (ii), upon assuming
	$$T\in [H^{s,p}_+]^\ast,
	$$
	we are required to find $$T_0\in L^{p'}\ \ \text{such that}\ \
	T=[\nabla^s_+]^\ast T_0\ \ \text{in}\ \ \cs'.
	$$
{Inspiring by \cite[Proposition~1, pp.\,399-400]{BB}, we} consider the operator
	\begin{align*}
	A_+:\ \     H^{s,p}_+\to L^p\ \ \text{via}\ \
	u\mapsto A_+u=(-\Delta)^{\frac s2}u
	\end{align*}
	Evidently, the above-defined linear operator $A_+$ is  bounded and hence closed. Thus,
	if $$u\in H^{s,p}_+\ \ \text{enjoys}\ \ \|(-\Delta)^{\frac s2}u\|_{L^p}=0,$$ then $$(-\Delta)^{\frac s2}u=0\ \ \text{almost everywhere on}\ \ \rn,
	$$
	and hence
	$$u=I_s(-\Delta)^{\frac s2}u\equiv0\ \ \text{on}\ \  \rn.
	$$
	This in turn implies that the operator $A_+$ is injective. Moreover, due to
	$$
	\|A_+u\|_{L^p}=\|(-\Delta)^\frac{s}{2}u\|_{L^p}=[u]_{H^{s,p}_+},
	$$
the operator $A_+$ has actually a continuous inverse from $L^p$ to $H^{s,p}_+$.
 Accordingly, by the closed range theorem (see \cite[p.\,208, Corollary~1]{Y}), we know that the adjoint operator
	$${ A}_+^\ast:\, L^{p'} \to [H^{s,p}_+]^\ast\ \ \text{
		defined by}\ \
	\laz { A}_+^\ast  F, u\raz=\laz  F, { A}_+ u\raz\ \ \forall\ \ (F,u)\in L^{p'}\times H^{s,p}_+,$$
	is surjective. In particular, if $$T\in [H^{s,p}_+]^\ast,
	$$
	then there exists $$T_0\in L^{p'}\ \ \text{such that}\ \
	{ A}_+^\ast  T_0 = T.$$
	Consequently, for any $\phi\in\cs$, we have
	\begin{align*}
	\laz { A}_+^\ast T_0,\, \phi\raz=\laz T_0, { A}_+ \phi\raz  =  \laz T_0,  (-\Delta)^\frac s2 \phi\raz
	=\laz [\nabla^s_+]^\ast T_0,   \phi\raz,
	\end{align*}
	namely,
	$$
	T={ A}_+^\ast T_0 = [\nabla^s_+]^\ast T_0 \ \ \text{in}\ \ \cs'.
	$$
This completes the argument for that (i) implies (ii).

Next, we show that (iii) implies (i)  by considering $H^{s,p}_-$.
If
	$$
	T=[\nabla^s_-]^\ast \vec{T}\ \ \text{in}\ \ \cs'\ \ \text{for some}\ \  \vec T=(T_1,\dots,  T_n)\in (L^{p'})^n,
	$$
	then for any $\phi\in\cs$ we have
	\begin{align*}
	\laz T, \phi\raz
	&=\laz [\nabla^s_-]^\ast \vec T, \phi\raz\\
&	=-\sum_{j=1}^n\laz (-\Delta)^\frac s2R_jT_j, \phi\raz\\
&=-\sum_{j=1}^n\laz R_jT_j, (-\Delta)^\frac s2\phi\raz\\
&=\sum_{j=1}^n\laz T_j,  R_j(-\Delta)^\frac s2\phi\raz\\
&	=\sum_{j=1}^n\laz   T_j, \nabla^s_j\phi\raz,
	\end{align*}
whence
	\begin{align*}
	|\laz T, \phi\raz| \le \sum_{j=1}^n \|T_j\|_{L^{p'}} \|\nabla^s_j\phi\|_{L^p} \ \ \forall\ \ \phi\in\cs.
	\end{align*}
	Since $\cs$ is dense in $H^{s,p}_-$, it follows that $T$ induces a bounded linear functional on $H^{s,p}_-$. This shows (iii)$\Longrightarrow$(i).

Conversely, in order to show (i)$\Longrightarrow$(iii), assuming $$T\in [H^{s,p}_-]^\ast,
	$$
	we are about to verify that
	$$
	T=[\nabla^s_-]^\ast \vec T=(T_1,\dots, T_n)\ \ \text{in}\ \ \cs'\ \ \text{for some}\ \  \vec T\in (L^{p'})^n.
	$$
	To this end, we consider the bounded linear operator
	\begin{align*}
	A_-:\  H^{s,p}_-\to (L^{p})^n\ \ \text{via}\ \ u\mapsto \nabla_-^su.
	\end{align*}
	Now we validate that the just-defined operator $A_-$ is injective.
	If $$u\in H^{s,p}_-\ \ \text{satisfies}\ \
	\nabla^s_- u=0\ \ \text{in}\ \ (L^{p})^n,
	$$
then, for any $\psi\in\cs_\infty$, we apply the Fourier transform to derive
$$ \psi=-\sum_{j=1}^n \nabla^s_j I_s R_j\psi\ \ \text{with}\ \ I_s R_j\psi\in\cs_\infty{\subseteq} L^{p'},$$
thereby giving
\begin{align*}
\lf|\laz u, \psi\raz\r|
=\lf|\sum_{j=1}^n \laz u,  \nabla^s_j I_s R_j\psi\raz\r|
=\lf|\sum_{j=1}^n \laz \nabla^s_j  u,  I_s R_j\psi\raz\r|
\le \sum_{j=1}^n \|\nabla^s_j  u\|_{L^p} \|I_s R_j\psi\|_{L^{p'}}
{=0.}
\end{align*}
This, along with the density of $\cs_\infty$ in $L^{p'}$ (cf. the proof of \cite[Lemma~2.9(iii)]{LXI}), further gives
$$u=0\ \ \text{in}\ \ L^p \ \& \ \text{a.\,e.}\ \ \Longrightarrow\ \  u=0\ \ \text{in}\ \  H^{s,p}_-.$$
Accordingly, $A_-$ is an injective map from
$H^{s,p}_-$ onto $A_-(H^{s,p}_-)$ (the closed range of $A_-$) $\subseteq (L^p)^n$. This, along with
$$
\|A_-u\|_{L^p}=\|\nabla^s_-u\|_{L^p}=[u]_{H^{s,p}_-},
$$
ensures that
$A_-$  has a continuous inverse from $A_-(H^{s,p}_-)$ to $H^{s,p}_-$. Upon applying the closed range theorem (see \cite[p.\,208, Corollary~1]{Y}) we get that the adjoint operator
$$
{A}_-^\ast:\, \big[A_-(H^{s,p}_-)\big]^\ast \to [H ^{s,p}_- ]^\ast\ \ \text{
	via}\ \
\laz { A}_-^\ast \vec F, u\raz=\laz \vec F, { A}_- u\raz\ \ \forall\ \  (\vec{F},u)\in \big[A_-(H^{s,p}_-)\big]^\ast\times H^{s,p}_-
$$
is surjective, thereby finding
$$
\vec{T}_o\in \big[A_-(H^{s,p}_-)\big]^\ast\ \ \text{such that}\ \ {A}_-^\ast \vec{T}_o=T.
$$
Upon utilizing the Hahn-Banach theorem to extend $\vec{T}_o$ to
		$$
	\vec T=(T_1,\dots, T_n)\in(L^{p'})^n=\big[(L^p)^n\big]^\ast
	$$
	we have
	\begin{align*}
\laz T,\phi\raz=\laz {A}_-^\ast \vec T_o,\, \phi\raz = \laz \vec T_o,  A_- \phi\raz =  \laz \vec T,  \nabla^s_- \phi\raz
	=\laz [\nabla^s_-]^\ast \vec T,   \phi\raz\ \ \forall\ \ \phi\in\cs,
	\end{align*}
	whence
	$$
	T={A}_-^\ast \vec T = {[\nabla^s_-]^\ast \vec T}\ \ \text{in}\ \ \cs'.
	$$
This completes the argument for (i)$\Longrightarrow$(iii).
\end{proof}

Let $\text{div}$ be the classical divergence operator whose action on a vector-valued function $\vec Y$ is given by
$$\text{div}\vec Y=\nabla\cdot \vec Y.$$
As a limiting case $s\uparrow1$ of Theorem \ref{t3.2}, we have the following conclusion.

\begin{proposition}	\label{p3.3}
Let $p\in(1,\infty)$. Then $L^p=\vec{R}\cdot(L^p)^n$ - namely -
		$$
		f\in L^p\Longleftrightarrow\exists\ (f_1,\dots, f_n)\in (L^p)^n\ \text{such that}\ f=\sum_{j=1}^n R_jf_j\ \ \text{in}\ \ L^p.
		$$
		Consequently, for any $Y\in L^p$, there exist $(Y_0,Y_1,\dots, Y_n)\in (L^p)^{1+n}$ such that
$$
\text{div}\big((-\Delta)^{-\frac12}Y_1,\cdots,(-\Delta)^{-\frac12}Y_n\big)=Y=(-\Delta)^\frac12Y_0\ \ \text{in}\ \ L^p.
$$
\end{proposition}

\begin{proof} Given $1<p<\infty$.
	Thanks to the boundedness of $\vec{R}$ on $L^{p}$ and the identity $$\vec{R}\cdot\vec{R}=-\text{id}\ \ \text{in}\ \  L^p,
	$$ we have that any
$f\in L^p$ enjoys the desired property
	$$
	f_j=-R_jf\in L^p\ \ \&\ \
	f=\sum_{j=1}^\infty R_j f_j\ \ \text{in}\ \ L^p.
	$$
	
As a consequence, for any $Y\in L^p$ we can find a vector-valued function
	$$
	\vec{Y}=(Y_1,\dots, Y_n)\in \big(L^p\big)^n
	$$
	such that
	$$
	 Y=\sum_{j=1}^nR_jY_j=\nabla\cdot\big((-\Delta)^{-\frac12}\vec{Y}\big)=\text{div}\Big((-\Delta)^{-\frac12}Y_1,\cdots,(-\Delta)^{-\frac12}Y_n\Big)
\ \ \text{in}\ \ \cs'.
	$$
	Also, if
	$$Y_0=I_1Y,$$
	then
	$$
	Y=(-\Delta)^\frac12 I_1 Y=(-\Delta)^\frac12 Y_0\ \ \text{in}\ \ \cs'.
	$$
Since $\cs$ is dense in $[L^p]^\ast=L^{\frac{p}{p-1}}$, we deduce that the last two equalities hold in $L^p$. 	
\end{proof}

\begin{remark}	\label{r3.4}
	Whenever $s=p=1$ we define
$$H^{1,1}=\lf\{f\in\cs_s':\, [f]_{H^{1,1}}
=\|(-\Delta)^\frac 12 f\|_{H^1}<\infty\r\}.$$
Just like $\cs_\infty$ is dense in $H^1$, we have also the density of $\cs_\infty$ in $H^{1,1}$ (cf. \cite[Proposition~2.12]{LXI}).
But for functions in $\cs_\infty$ the Fourier transform easily derives
$$
[f]_{H^{1,1}}
=\|(-\Delta)^\frac 12 f\|_{L^1}+\|\nabla f\|_{L^1}.
$$
Thus, $H^{1,1}$ can be equivalently defined to be the space of all locally integrable functions on $\rn$ satisfying
$[f]_{H^{1,1}}<\infty$.
In analogy to {Theorem \ref{t3.2} and Proposition \ref{p3.3}}, we have:
	\begin{itemize}
		\item[\rm (i)] $H^{1,1}=\vec{R}\cdot\big(H^{1,1}\big)^n$ - namely -
		$$
		Z\in H^{1,1}\Longleftrightarrow
		\exists\ (Z_1,\dots, Z_n)\in (H^{1,1})^n\ \text{such that}\ Z=\sum_{j=1}^n R_jZ_j.
		$$
		This is due to the fact that any $Z\in H^{1,1}$ can be written as
		$$
		Z=\sum_{j=1}^n R_jZ_j\ \ \text{where}\ \ Z_j=-R_jZ\in H^{1,1}.
		$$
		
		\item[\rm (ii)] Given a distribution $T\in\mathcal{S}'$,
		\begin{equation*}
		T\in [H^{1,1}]^\ast\Longleftrightarrow \exists\ (T_0,T_1,\dots, T_n)\in (L^\infty)^{1+n}\ \text{such that}\ T=(-\Delta)^\frac12 T_0-\text{div}(T_1,\dots,T_n)\ \text{in}\ \mathcal{S}'.
		\end{equation*}
		This follows from the endpoint $s=1$ of \cite[Theorem 4.3(i)]{LXI} (cf. \cite[Lemma 4.1]{PT1} for the dual of the endpoint Sobolev space $\dot{W}^{1,1}$) and the basic formula
		$$
		[\nabla^1_+]^\ast=(-\Delta)^\frac12\ \ \&\ \ [\nabla^1_-]^\ast=-\text{div}.
		$$
		
		\item[\rm (iii)] Thanks to (i) and the fact that any  $$\vec{Z}=(Z_1,\dots, Z_n)\in \big(H^{1,1}\big)^n$$
		satisfies
		$$
		 \sum_{j=1}^nR_jZ_j=\nabla\cdot\big((-\Delta)^{-\frac12}\vec{Z}\big)=\text{div}\Big((-\Delta)^{-\frac12}Z_1,\dots, (-\Delta)^{-\frac12}Z_n\Big),
		$$
we get that
$$
		\forall\ Z\in H^{1,1}\
		\exists\ (Z_1,\dots, Z_n)\in (H^{1,1})^n\ \text{such that}\ \text{div}\big((-\Delta)^{-\frac12}Z_1,\dots, (-\Delta)^{-\frac{1}{2}}Z_n\big)=Z.
		$$		
	\end{itemize}
\end{remark}


\section{Distributional solutions of duality equations}\label{s4}

\subsection{Distributional solutions to $[\nabla^s_\pm]^\ast u=\mu$}\label{s4.1}

{}{For any} $\alpha\in(0,n)$ and  nonnegative Radon measure $\mu$ on $\rn$, define
$$
{I_\alpha\mu(x)=c_{n,\alpha}}\int_{\rn}|x-y|^{\alpha-n}\,d\mu(y)\qquad\forall\ \ x\in\rn
$$
and
$$
\||\mu\||_{n-\alpha}=\sup_{(x,r)\in\rn\times(0,\infty)}r^{\az-n}\mu\big(B(x,r)\big).
$$
Observe that
	\begin{equation}\label{e51}
	I_\alpha\mu(x){\ge c_{n,\alpha}\int_{B(0,r)}|x-y|^{\alpha-n}\,d\mu(y)\ge c_{n,\alpha}}\mu\big(B(0,r)\big)\big(|x|+r\big)^{\alpha-n}\ \ \ \forall\ \ (x,r)\in\rn\times(0,\infty).
	\end{equation}
As a straightforward application of Theorem \ref{t3.2}, we can characterize distributional solutions to the following
fractional  duality equations
$$[\nabla^s_+]^\ast u_0=\mu\ \ \&\ \
[\nabla^s_-]^\ast(u_1,\dots, u_n)=\mu.$$

Upon extending \cite[Theorems 3.1-3.2-3.3]{PT} - if $\mu$ is a nonnegative Radon measure on $\mathbb R^{n\ge 2}$ then
$$
\begin{cases}
\exists\
\vec{F}\in{(L^{\frac n{n-1}<p<\infty})^n}\ \ \text{such that}\ \ \text{div}\vec{F}=\mu\Longleftrightarrow
I_1\mu\in L^p\\
\exists\ \vec{F}\in \big(L^{{1\le p}\le\frac{n}{n-1}}\big)^n\ \ \text{such that}\ \ \text{div}\vec{F}=\mu\Longleftrightarrow \mu=0\\
\exists\ \vec{F}\in(L^\infty)^n\ \text{such that}\ \text{div}\vec{F}=\mu\Longleftrightarrow \||\mu\||_{n-1}<\infty,
\end{cases}
$$
we obtain

\begin{theorem}\label{t4.1}
Let $0<s<1<p<\infty$ and $\mu$ be a nonnegative Radon measure on $\mathbb R^{n}$.
Then either
\begin{align}\label{eq-z1}
\exists\ u_0\in L^p \ \text{such that}\ \ [\nabla^s_+]^\ast u_0=\mu \ \ \text{in}\ \ \cs'
\end{align}
or
\begin{align}\label{eq-z2}
\exists\ (u_1,\dots, u_n)\in (L^p)^{n}\ \text{such that}\ [\nabla^s_-]^\ast(u_1,\dots, u_n)=\mu \ \ \text{in}\ \ \cs'
\end{align}
holds if and only if
$$
\begin{cases}
\mu=0 &\quad\textup{if}\ \; p\in(1,\frac n{n-s}]\\
{I_s\mu}\in L^{p} &\quad\textup{if}\ \; p\in(\frac n{n-s},\infty).
\end{cases}
$$

\end{theorem}

\begin{proof}
Let us start with the case $p\in(1, \frac n{n-s}]$. Clearly, if $\mu=0$, then  $$u_0=u_1=\cdots=u_n=0$$ ensures
$$[\nabla^s_+]^\ast u_0=0$$
and
$$[\nabla^s_-]^\ast(u_1,\dots, u_n)=0.$$
Thus it is enough to show the only-if-part.

Consider first the operator $[\nabla^s_+]^\ast$ and assume that  \eqref{eq-z1} holds for some $u_0\in L^p$. For any $\phi\in\cs_\infty$, we utilize the Fourier transform to derive
\begin{align*}
\phi=(-\Delta)^\frac s2 I_s\phi
\end{align*}
and hence
\begin{align*}
\laz u_0,\phi\raz
&=\laz u_0, (-\Delta)^\frac s2 I_s\phi\raz
=\laz [\nabla^s_+]^\ast u_0,\, I_s\phi\raz
=\int_\rn {I_s\phi(x)}\,d\mu(x)
=\int_\rn \big({I_s\mu(x)} \big)\phi(x)\, dx,
\end{align*}
which, along with the fact that $\cs_\infty$ is dense in $$[L^p]^\ast=L^{\frac p{p-1}},$$ gives
$$
{I_s\mu}=u_0\ \ \text{in}\ \ L^p.
$$
From this and the observation \eqref{e51} it follows that
	 $$
	 \int_{\rn}\lf(\mu\big(B(0,r)\big)(|x|+r)^{(s-n)}\r)^p\,dx<\infty\quad\text{under}\quad (n-s)p\le n.
	 $$
	 However, this is impossible unless $\mu=0$.

Consider next the operator $[\nabla^s_-]^\ast$. Assume that  \eqref{eq-z2} holds - namely - $$\vec{u}=(u_1,\dots, u_n)\in (L^p)^n$$ is a distributional solution of
$$
{[\nabla^s_-]^\ast \vec u=\mu.}
$$
For any $\psi\in\cs_\infty$, by the fact $I_s\psi\in\cs_\infty$, the definition of
$${[\nabla^s_-]^\ast=-\text{div}^s=}-(-\Delta)^\frac s2\vec R$$ and  the self-adjointness of $(-\Delta)^\frac s2$, we obtain
\begin{align*}
\laz I_s\mu,\psi\raz
&=\lf\laz I_s\lf([\nabla^s_-]^\ast \vec u\r), \psi\r\raz\\
&=\laz [\nabla^s_-]^\ast \vec u, I_s\psi\raz\\
&=-\int_\rn \text{div}^s\vec{u}(x)I_s\psi(x)\,dx\\
&=-\sum_{j=1}^n\int_\rn (-\Delta)^\frac s2 R_j u_j(x)I_s\psi(x) \,dx\\
&=-\sum_{j=1}^n\int_\rn  R_j u_j(x)(-\Delta)^\frac s2I_s\psi(x)\,dx\\
&=-\sum_{j=1}^n\int_\rn R_j u_j(x)\psi(x)\,dx\\
&=-\sum_{j=1}^n \laz R_j u_j,\psi\raz,
\end{align*}
which, together with the {aforementioned} density of $\cs_\infty$ in $$[L^p]^\ast=L^{\frac p{p-1}}
$$ and the boundedness of $R_j$ on $L^p$, yields
$$I_s\mu=-\sum_{j=1}^n R_j u_j\ \ \text{in}\ \ L^p.$$
Similarly to the argument for the operator $[\nabla^s_+]^\ast$, the fact $I_s\mu\in L^p$ and \eqref{e51} again derive $\mu=0$.

Next, we handle the case $p\in(\frac n{n-s},\infty)$.  Clearly, the only-if-part follows from the same argument as the case $p\in(1,\frac{n}{n-s}]$.
So, it remains to verify the if-part under
$$
I_s\mu\in L^p\ \ \text{for}\ \ (n-s)p>n.
$$
According to Theorem \ref{t3.2}, we only need to validate that  such a measure $\mu$ induces a bounded linear functional on $H^{s, p'}_+$, where $p'=\frac p{p-1}$.
To this end, for any $\phi\in\cs$, by the fact
$$\phi=I_s(-\Delta)^\frac s2 \phi$$and the Fubini theorem, we write
\begin{align*}
\int_\rn\phi\,d\mu
=\int_\rn I_s(-\Delta)^\frac s2 \phi(x)\, d\mu(x)
=\int_\rn (-\Delta)^\frac s2 \phi(x) I_s\mu(x)\,dx,
\end{align*}
so the H\"older inequality gives
\begin{align*}
\lf|\int_\rn\phi\,d\mu\r|\le \|I_s\mu\|_{L^p}\|(-\Delta)^\frac s2 \phi\|_{L^{p'}}
=\|I_s\mu\|_{L^p} [\phi]_{H^{s, p'}_+}.
\end{align*}
Combining this with the density of $\cs$ in $H^{s, p'}_+$ (cf. Theorem \ref{t3.1}) leads to that
$\mu$ can be extended to a bounded linear functional on $H^{s, p'}_+$.

\end{proof}


\subsection{Morrey's regularity for distributional solutions of $[\nabla^s_\pm]^\ast u=f$}\label{s4.2}
In accordance with
the basic identity
$$
[\nabla^s_-]^\ast(\nabla^s_- u)=-[\nabla^{2s}_+]^\ast u\ \ \ \forall\ \ \ u\in C_c^\infty
$$
and
\cite[Theorem 1.1]{SSS} - if $\Omega$ is an open subset of $\rn$,
$$
(p,s)\in (2-n^{-1},\infty)\times (0,1],
$$
and $u\in H^{s,p}$ is a distributional solution to the following fractional $p$-Laplace equation with a natural variation structure
$$
\text{div}^s(|\nabla^s_-u|^{p-2}\nabla^s_-u)=0\ \ \text{in}\ \ \Omega,
$$
i.e.,
$$
\int_{\rn}|\nabla^s_-u|^{p-2}\nabla^s_- u\cdot\nabla^s_-\phi\,dx=0\ \ \forall\ \ \phi\in C^\infty_c(\Omega),
$$
then $u\in C^{s+\alpha}_{\loc}(\Omega)$ for some positive constant $\alpha$ depending on $p$ only, we are led
to settle Morrey's regularity for the distributional solutions of the fractional duality equations
$$
[\nabla^s_\pm]^\ast u=f.
$$

For any $(p,\kappa)\in [1,\infty)\times(0,n]$,
the Morrey space $\mathrm{L}^{p,\kz}$ was introduced by Morrey \cite{Morrey} and used to study
the solution of some quasi-linear elliptic partial differential equations,
where $\mathrm{L}^{p,\kz}$ comprises
all  Lebesgue measurable functions $f$ on $\mathbb R^n$ with
 $$
 \|f\|_{\mathrm{L}^{p,\kappa}}=\sup_{(x,r)\in\mathbb R^n\times(0,\infty)}\left(r^{\kappa-n}\int_{B(x,r)}|f(y)|^p\,dy\right)^\frac1p<\infty.
 $$
In particular, when $(p,\kz)\in [1,\infty)\times\{n\}$, the space $\mathrm{L}^{p,n}$
is just the classical Lebesgue space $L^p$.

For $(p,\kappa)\in (1,\infty)\times(0,n)$,
let $\mathrm{H}^{p,\kz}$ be the space of
all Lebesgue measurable functions
$f$ on $\rn$ such that
$$
\|f\|_{\mathrm{H}^{p,\kz}}=\inf_\omega\lf(\int_\rn|f(x)|^{p}
\big(\omega(x)\big)^{1-p}\,dx\r)^\frac1p<\fz,
$$
where the infimum is taken over all nonnegative functions
$\omega$ on $\rn$ satisfying
$$
\|\omega\|_{L^1(\Lambda^{n-\kappa}_{(\infty)})}= \int_0^\fz\Lambda^{n-\kappa}_{(\infty)}
\lf(\lf\{x\in\rn:\ \omega(x)>t\r\}\r)\,dt\leq1.
$$
Here and hereafter, for any given $\az\in(0,n)$, the symbol $\Lambda^{\az}_{(\infty)}(E)$ denotes
the $\az$-th order Hausdorff capacity
of a subset $E\subseteq\rn$, given by
$$\Lambda^\az_{(\fz)}(E)=\inf\lf\{\sum_j r_j^\az:\ E\subseteq\bigcup_j B(x_j,r_j)
\ {\rm{with}}\ x_j\in\rn\ {\rm{and}}\ r_j\in(0,\fz)\r\}.$$
According to  \cite{ax04}, we have {the duality}
\begin{equation*}
\label{dual}
[\mathrm{H}^{p',\kz}]^\ast={}{\mathrm{L}^{p,\kz}}.
\end{equation*}
From \cite[(5.1)]{PS} and {\cite[Corollary \& Proposition~5]{Anote}}, we have that if $$\||\mu\||_{n-\kappa}<\infty$$
then
\begin{equation}
\label{eSWH0}
\int_{\rn}|I_\kappa u|\,d\mu\lesssim \|I_\kappa u\|_{L^1(\Lambda^{n-\kappa}_{(\infty)})}\lesssim \|u\|_{H^1}\ \ \forall\ \ u\in H^1.
\end{equation}
Consequently, if
$$
d\nu_\kappa(x)=|x|^{-\kappa}\,dx
$$
then
$$
\||\nu_\kappa\||_{n-\kappa}<\infty
$$
and hence \eqref{eSWH0} is used to produce the Stein-Weiss-Hardy inequality at the endpoint $p=1$:
\begin{equation}
\label{eSWH}
\int_{\rn}|x|^{-\kappa}|I_\kappa u(x)|\,dx\lesssim \|u\|_{H^1}\ \ \forall\ \ u\in H^1.
\end{equation}
This, along with (cf. \cite[(1.3)-(1.4)]{LXI})
\begin{align*}\label{e2a}
[u]_{H^{s,1}}\ls [u]_{{W}^{s,1}}=\int_{\mathbb R^n}\int_{\mathbb R^n}\frac{|u(x)-u(y)|}{|x-y|^{n+s}}\,dy\,dx\ \ \forall\ \ u\in{\cs},
\end{align*}
derives
\begin{equation}
\label{e1H}
\int_{\rn}|x|^{-s}|u(x)|\,dx\lesssim [u]_{H^{s,1}}\lesssim [u]_{{W}^{s,1}}\ \ \forall\ \ u\in\cs,
\end{equation}
which may be viewed as an improvement of the case $p=1$ of \cite[Theorem 1.1]{FrS}.

Upon taking a function $\varphi$ satisfying
\begin{align*}
\begin{cases}
0\le\varphi\in\cs\\
\int_\rn \varphi(x)\, dx=1\\
\varphi_t(x)=t^{-n}\varphi(t^{-1}x)\ \ \forall\ \ (t,x)\in(0,\infty)\times\rn,
\end{cases}
\end{align*}
we extend the real Hardy space $H^p$ from $p\in [1,\infty)$ to $p\in (0,\infty)$ via defining (cf. \cite{St2})
$$
\mathrm{H}^p=\lf\{
f\in\cs':\  \|f\|_{\mathrm{H}^p}=\Big\|\sup_{t\in(0,\infty)} |\varphi_t\ast f|\Big\|_{L^p}<\infty
\r\}\ \ \text{under}\ \  0<p<\infty.
$$
Then {(cf. \cite{FS, St2})}
$$
[\mathrm{H}^p]^\ast
=\begin{cases}
\mathrm{BMO} \ \ &\text{as}\ \ p=1\\
\mathrm{Lip}_{n(p^{-1}-1)}\ \ &\text{as}\ \ p\in\big(\frac n{n-1}, 1\big).
\end{cases}
$$
Here and henceforth, $\mathrm{Lip}_{0<\alpha<1}$ is the $(0,1)\ni\alpha$-Lipschitz space of all functions $f$ on $\rn$ satisfying
$$
\|f\|_{\mathrm{Lip}_{\az}}=\sup_{x,y\in\rn,\, x\neq y}\frac{|f(x)-f(y)|}{|x-y|^{\az}}<\infty.
$$

\begin{theorem}\label{t4.3}
Let
\begin{align*}
\begin{cases}
0<s<1<n\\
0<\kz\le n\\
1\le p<\frac{\kappa}{s}\\
1<q
<\frac n{\frac\kz p-s}\\
1<\begin{cases}
q\le \frac{\kz }{\frac\kz p-s} \ \ &\textup{as}\ 1<p<\frac \kappa s\\
q<\frac{n}{n-s}\ \ &\textup{as}\ 1=p<\frac \kappa s.
\end{cases}
\end{cases}
\end{align*}
If $f\in \mathrm{L}^{p,\kz}$, then
$$
\exists\ {(F_0, F_1,\dots, F_n)}\in
\begin{cases}
\big(\mathrm{Lip}_{s-\frac\kz p}\big)^{1+n} \quad&\textup{as}\; f\in \mathrm{L}^{p>\frac{\kz}s,\kz}\;\\
\big(\mathrm{BMO}\big)^{1+n}\quad&\textup{as}\; f\in \mathrm{L}^{p=\frac{\kz}{s},\kappa}\;\\
\big(\mathrm{L}^{q,\,q(\frac \kz p-s)}\big)^{1+n}\quad&\textup{as}\; f\in \mathrm{L}^{p<\frac{\kz}{s},\kappa}\;
\end{cases}
$$
such that
$$
[\nabla^s_+]^\ast F_0 =f=[\nabla^s_-]^\ast(F_1,...,F_n)
$$
holds in the sense of
$$
\int_\rn \Big([\nabla^s_+]^\ast F_0-f\Big)(x)\phi(x)\, dx=0=
\int_\rn \Big([\nabla^s_-]^\ast \big(F_1,...,F_n\big)-f\Big)(x)\phi(x)\, dx\ \ \forall\;\phi\in \cs.
$$
\end{theorem}

\begin{proof} Suppose $f\in \mathrm{L}^{p,\kz}$. Note that
the desired regularity for
$$
[\nabla_+^s]^\ast F_0=(-\Delta)^\frac{s}{2}F_0=f\ \ \text{in}\ \ \mathcal{S}'
$$
follows from \cite[Theorem 1.2]{LXaSNS} with
$F_0=I_s f$. So, it remains to check the desired regularity for
$$
[\nabla_-^s]^\ast (F_1,...,F_n)=f\ \ \text{in}\ \ \mathcal{S}'.
$$

To this end, we define the measure
$\mu_f$ by $$d\mu_f(x)=|f(x)|\, dx.$$
Then, for any $(x,r)\in\rn\times (0,\infty)$, we utilize the H\"older inequality to derive
$$
\mu_f(B(x,r))=\int_{B(x,r)} |f(y)|\, dy
\le \lf( \int_{B(x,r)} |f(y)|^p\, dy\r)^\frac 1p |B(x,r)|^\frac{p-1}{p}
\ls \|f\|_{\mathrm{L}^{p,\kz}} r^{n-\frac\kz p},
$$
thereby achieving
$$
\||\mu_f\||_{n-\frac\kz p}\ls \|f\|_{\mathrm{L}^{p,\kz}}<\infty.
$$

The forthcoming demonstration consists of essentially two components.
	
{\it Part 1 - the case $sp\ge\kz$}.

Under this condition we have
$$
\lf[\mathrm{H}^{\frac n{n+s-\frac\kz p}}\r]^\ast=\begin{cases}
\BMO\ \ \text{as}\ \ sp=\kz\\
\mathrm{Lip}_{s-\frac\kz p}\ \ \text{as}\ \ sp>\kz.
\end{cases}
$$
We are inspired by the proof of \cite[Proposition~1, pp.\,399-400]{BB} (cf. \cite[Theorem~3.2]{PT}) to set
$$Y= (\mathrm{H}^{\frac n{n+s-\frac\kz p}})^n=\overbrace{\mathrm{H}^{\frac n{n+s-\frac\kz p}}\times \cdots\times \mathrm{H}^{\frac n{n+s-\frac\kz p}}}^n$$
and
$$
X=\lf\{u\in \cs_s':\, \nabla^s_ju\in \mathrm{H}^{\frac n{n+s-\frac\kz p}}\;\textup{for}\; j=1,2,\dots,n\r\},
$$
endowed with the norm
$$\|u\|_{X}=\sum_{j=1}^n\|\nabla^s_ju\|_{\mathrm{H}^{\frac n{n+s-\frac\kz p}}}.$$
Note that $\|u\|_{X}=0$ if and only if $u$ is a constant function on $\rn$. So, $X$ is treated as a quotient space  modulo the space of  constant functions. Since $\cs_\infty\subset\cs\cap X\subset X$ and $\cs_\infty$ is dense in the Hardy space $\mathrm{H}^{\frac n{n+s-\frac\kz p}}$ (cf. \cite{Anote}), one easily deduces the density of $\cs\cap X$ in $X$.

Consider the operator
\begin{align*}
A:\  X\to Y \ \  \text{via}\ \ u\mapsto A(u)={\nabla^s_- u}.
\end{align*}
This operator is well defined in that the action of the operator $\nabla^s_-$ can be defined on the distribution space $\cs_s'$.
Moreover, it is easy to see that $A$ is a bounded  linear operator.

We can also show that the operator $A$ is injective. To this end, assuming  that $u\in X$ satisfies
	$$\nabla^s_- u=0\ \ \text{in}\ \ (\mathrm{H}^{\frac n{n+s-\frac\kz p}})^n,
	$$
	we are required to show
	$$
	u=\text{constant}\ \ \Longrightarrow\ \ u= 0\ \ \text{in}\ \  X.
	$$
Note that
$$u\in X\Longrightarrow u\in\cs_s'\ \ \&\ \ \nabla^s_ju\in \mathrm{H}^{\frac n{n+s-\frac\kz p}}.
$$
Thus, for any $\psi\in\cs_\infty$, we use the Fourier transform to derive
$$ \psi=-\sum_{j=1}^n \nabla^s_j I_s R_j\psi\ \ \text{with}\ \ I_s R_j\psi\in\cs_\infty\subseteq \mathrm{Lip}_{s-\frac\kz p},$$
thereby finding
\begin{align*}
\lf|\laz u, \psi\raz\r|
=\lf|\sum_{j=1}^n \laz u,  \nabla^s_j I_s R_j\psi\raz\r|
=\lf|\sum_{j=1}^n \laz \nabla^s_j  u,  I_s R_j\psi\raz\r|
\le \sum_{j=1}^n \|\nabla^s_j  u\|_{\mathrm{H}^{\frac n{n+s-\frac\kz p}}} \|I_s R_j\psi\|_{\mathrm{Lip}_{s-\frac\kz p}}
=0.
\end{align*}
This shows
$$u=0\ \ \text{in}\ \ \cs'/\mathcal P.
$$
In other words, $u$ is a polynomial on $\rn$.
However, if a polynomial $u$ is a bounded linear functional on $\cs_s$,
then $u$ must be a constant function, as desired.

The above analysis shows that the operator $A$ is injective and has a continuous inverse from $A(X)\subseteq Y$ to $X$.
Upon applying the closed range theorem (see \cite[p.\,208, Corollary~1]{Y}), we deduce that the adjoint operator
$$ A^\ast:\, [A(X)]^\ast\to X^\ast\ \ \text{
		via}\ \
	\laz { A}^\ast \vec F, u\raz=\laz \vec F, {A} u\raz\ \ \forall\ \  (\vec{F},u)\in [A(X)]^\ast\times X.$$
is surjective.

Next, we validate that any $f\in \mathrm{L}^{p,\kz}$ belongs to $X^\ast$.
Indeed, for any $\phi\in \cs\cap X$, we apply \cite[Theorem~1.12]{SS1} to write
\begin{align}\label{ss1-t1.12}
\phi=I_s\lf(\sum_{j=1}^n R_j \nabla^s_j\phi\r).
\end{align}
Also, using $\phi\in\cs$, we derive from \cite[Lemma~2.6]{LXI} that
$\nabla^s_j\phi\in\cs_s$,
which easily implies that
$R_j \nabla^s_j\phi$ is continuous on $\rn$.
From the fact
$$\frac\kz p \le s<1\le n-1
$$
it follows that
\begin{align}\label{eq-z3}
\frac{n}{n+s-\frac\kz p}<\frac ns\ \ \&\ \
n-\frac{sn}{n+s-\frac\kz p}\le n-\frac\kz p,
\end{align}
while the second inequality of \eqref{eq-z3} holds because after a change of variable $$0<t=\frac\kz p\le s$$ the function
$$\psi(t)=t(n+s-t)-sn$$
is strictly increasing on the interval $(0,s]$ and $\psi(s)=0$.
By \eqref{eq-z3}, \cite[Theorem~1.1]{LX-jfa} and its remark, we can derive the continuity of the mapping
$$
I_s:\, \mathrm{H}^{\frac n{n+s-\frac\kz p}}\cap\{\text{all continuous functions}\}\to L_{\mu_f}^{1},
$$
with operator norm at most a constant multiple of $\||\mu_f\||_{n-\frac \kz p}$.
Combining these and
boundedness of $R_j$ on $\mathrm{H}^{\frac n{n+s-\frac\kz p}}$ yields
\begin{align*}
\lf|\int_\rn \phi(x) f(x)\, dx\r|
&= \lf| \sum_{j=1}^{n}\int_\rn I_s \lf( R_j \nabla^s_j\phi\r)(x) f(x)\,dx\r|\\
&\le \sum_{j=1}^{n} \int_\rn \lf|I_s \lf( R_j \nabla^s_j\phi\r)(x)\r|\,d\mu_f(x)\\
&\ls \sum_{j=1}^{n} \||\mu_f\||_{n-\frac \kz p} \|R_j \nabla^s_j\phi\|_{\mathrm{H}^{\frac n{n+s-\frac\kz p}}}\\
&\ls \|f\|_{\mathrm{L}^{p,\kz}}\sum_{j=1}^{n}  \|\nabla^s_j\phi\|_{\mathrm{H}^{\frac n{n+s-\frac\kz p}}}\\
&\approx \|\phi\|_X \|f\|_{\mathrm{L}^{p,\kz}}.
\end{align*}
Due to the density of $\cs\cap X$ in $X$,  we arrive at the conclusion that
 $f$ induces a bounded linear functional on $X$.

To continue, like proving Theorem \ref{t3.2}(iii) we use the  surjective property  of $A^\ast$ and the Hahn-Banach extension theorem to obtain
$$\vec F=(F_1,\dots, F_n)\in Y^\ast=(\mathrm{Lip}_{s-\frac\kz p})^n
$$
such that
	\begin{align*}
\laz f,\, \phi\raz=	\laz { A}^\ast \vec F,\, \phi\raz=\laz \vec F, { A} \phi\raz  =  {\laz \vec F,  \nabla^s_- \phi\raz
	=\laz [\nabla^s_-]^\ast \vec F,   \phi\raz}\ \ \forall\ \ \phi\in\cs,
	\end{align*}
	whence
	$$
[\nabla^s_-]^\ast \vec F= { A}^\ast \vec F=f\ \ \text{in}\ \ \cs'.
	$$

{\it Part 2 - the case $sp<\kz$}.

This part is similar to the case $sp\ge\kz$. To be precise, we take
$$Y= \big(\mathrm{H}^{q',q(\frac \kz p-s)}\big)^n.$$
Define
$$
X=\lf\{u\in \cs_s':\, \nabla^s_ju\in \mathrm{H}^{q',q(\frac \kz p-s)}\;\textup{for}\; j=1,2,\dots,n\r\}
$$
endowed with the norm
$$\|u\|_{X}=\sum_{j=1}^n\|\nabla^s_ju\|_{\mathrm{H}^{q',q(\frac \kz p-s)}}.$$
Again, observing that $\|u\|_{X}=0$ if and only if $u$ is a constant, we also understood this $X$ as a quotient space.
Though we do not know if $\cs\cap X$ is dense in  $X$, we use the space $\mathring X$ which is the closure of $\cs\cap X$ in $X$.

Still we consider the operator
\begin{align*}
A:\  \mathring X\to Y \ \  \text{via}\ \ u\mapsto A(u)={\nabla^s_- u},
\end{align*}
and can show that $A$ is injective and has a continuous inverse from $A(\mathring X)$ (the close range of $A$) to $\mathring X$. Consequently, the closed range theorem (cf. \cite[p.\,208, Corollary~1]{Y}) can be applied to derive that the adjoint operator
$$ A^\ast:\, \big[A(\mathring X)\big]^\ast\to (\mathring X)^\ast\ \ \text{
		via}\ \
	\laz { A}^\ast \vec F, u\raz=\laz \vec F, { A} u\raz\ \ \forall\ \  (\vec{F},u)\in \big[A(\mathring X)\big]^\ast\times \mathring  X$$
is surjective.

Next, we validate that any $f\in \mathrm{L}^{p,\kz}$ belongs to $(\mathring X)^\ast$. Applying \cite[Proposition~5.1]{LX00} gives
the continuity of the mapping
$$I_s: \ \mathrm{L}^{p,\kz}\to \mathrm{L}^{q, q(\frac \kz p-s)}.$$
Note that the boundedness of $R_j$ on  ${\mathrm{H}^{q', q(\frac \kz p-s)}}$ is given in \cite[Chapter~8]{Abook}.
So, upon using \eqref{ss1-t1.12} and the Fubini theorem, we derive that any $\phi\in\cs\cap X$ satisfies
\begin{align*}
\lf|\int_\rn \phi(x) f(x)\, dx\r|
&= \lf| \sum_{j=1}^{n}\int_\rn I_s \lf( R_j \nabla^s_j\phi\r)(x) f(x)\,dx\r|\\
&= \lf| \sum_{j=1}^{n}\int_\rn   R_j \nabla^s_j\phi(x) I_sf(x)\,dx\r|\\
&\le \sum_{j=1}^{n}  \|R_j \nabla^s_j\phi\|_{{\mathrm{H}^{q', q(\frac \kz p-s)}}}\|I_sf\|_{\mathrm{L}^{q, q(\frac \kz p-s)}}\\
&\ls \sum_{j=1}^{n}  \|\nabla^s_j\phi\|_{{\mathrm{H}^{q', q(\frac \kz p-s)}}}\|f\|_{\mathrm{L}^{p,\kz}}\\
&\approx \|\phi\|_X \|f\|_{\mathrm{L}^{p,\kz}}.
\end{align*}
This implies that $f$ can be extended to a bounded linear functional on $\mathring X$, that is, $f\in (\mathring X)^\ast$.

Because of $f\in(\mathring X)^\ast$ and the surjective property of $A^\ast$, we can borrow the idea of verifying Theorem \ref{t3.2}(iii) and use the Hahn-Banach extension theorem to find a vector-valued function
$$\vec F=(F_1,\dots, F_n)\in Y^\ast= \big(\mathrm{L}^{q,\,q(\frac \kz p-s)}\big)^n$$
such that
\begin{align*}
	\laz { A}^\ast \vec F,\, \phi\raz=\laz \vec F, { A} \phi\raz  = { \laz \vec F,  \nabla^s_- \phi\raz
	=\laz [\nabla^s_-]^\ast \vec F,   \phi\raz}\ \ \forall\ \ \phi\in\cs,
	\end{align*}
	thereby reaching
	$$
[\nabla^s_-]^\ast \vec F= { A}^\ast \vec F=f\ \ \text{in}\ \ \cs'.
	$$
\end{proof}


\begin{thebibliography}{99}
	
\bibitem{ADuke} D.R. Adams, A note on Riesz potentials. {\it Duke Math. J.} 42(1975)765-778.
	
\bibitem{Anote} D.R. Adams, A note on Choquet integrals with respect to Hausdorff capacity. {\it Lecture Notes in Mathematics} 1302(1980)115-124. Springer.
	

\bibitem{Adams88} D.R. Adams,  A sharp inequality of J. Moser for higher order derivatives.
\emph{Ann. of Math. (2)} 128(1988)385-398.

\bibitem{Abook} D.R. Adams, \emph{Morrey Spaces.} Lecture Notes in Applied and Numerical Harmonic Analysis. Birkh\"auser/Springer, Cham, 2015.


\bibitem{ax04} D.R. Adams and J. Xiao, Nonlinear potential analysis
on Morrey spaces and their capacities. \emph{Indiana Univ. Math. J.} 53(2004)1629-1663.


\bibitem{Be} W. Beckner, {Pitt's inequality with sharp convolution estimates}. {\it Proc. Amer. Math. Soc.} 136(2008)1871-1885.


\bibitem{BB}
J.  Bourgain and H. Brezis,
On the equation $\text{div} Y=f$ and application to control of phases. \emph{J. Amer. Math. Soc.} 16(2003)393-426.



\bibitem{Bucur} C. Bucur, Some observations on the Green function for the ball in the fractional Laplace framework.
\emph{Comm. Pure Appl. Anal.} 15(2016)657-699.


\bibitem{CS} G.E. Comi and G. Stefani. A distributional approach to fractional Sobolev spaces and fractional variation: existence of blow-up. {Preprint}, Sept. 23, 2018.


\bibitem{FS} C. Fefferman and E.M. Stein, $H^p$ spaces of several variables. {\it Acta Math.}129 (1972)137-193.

\bibitem{FrS} R.L. Frank and R. Steiringer, Non-linear ground state representations and sharp Hardy inequalities.  {\it J. Funct. Anal.} 255(2008)3407-3430.

\bibitem{Fu} B. Fuglede, The logarithmic potential in higher dimensios. {\it Mat. Fys. Medd. Dan. Cid. Selsk.} 33:1(1960)1-14.

\bibitem{GIT} D.V. Gorbachev, V.I. Ivanov and S.Yu. Tikhonov, Riesz potential and sharp function for Dunkl transform. {\it arXiv:1708.09733v1[math.CA]31Aug2017.}

\bibitem{HLP} G. Hardy, J.E. Littlewood and G. P\'olya, {\it Inequalities, 2nd ed.},
Cambridge, Cambridge University Press, 1952.

\bibitem{Her} I.W. Herbst, Spectral theory of the operator $(p^2+m^2)^\frac12-Ze^2/r$. {\it Comm. Math. Phys.} 53(1977)285-294.


\bibitem{JN} F. John and L. Nirenberg,
On functions
of bounded mean oscillation. {\it
Comm. Pure Appl. Math.}
14(1961)415-426.



\bibitem{LX-jfa}  L. Liu and J. Xiao,
A trace law for the Hardy-Morrey-Sobolev space. {\it J. Funct. Anal.} 274(2018)80-120.

\bibitem{LXaSNS} L. Liu and J. Xiao, Morrey potentials from Campanato classes.  {\it Ann. Sc. Norm. Super. Pisa Cl. Sci.} (5)18(2018)1503-1517.



\bibitem{LX00}
L. Liu and J. Xiao,
Fractional Morrey integrals in Campanato-Sobolev spaces and $\textup{div}F=f$, submitted.


\bibitem{LXI}
L. Liu and J. Xiao, Intrinsic nature of the Stein-Weiss $H^{1}$-inequality, submitted.

\bibitem{MMW} M.M. Meerschaert, J. Mortensen and S.W. Wheatcraft, Fractional vector calculus for fractional advection-dispersion. {\it Physica A} 367(2006)181-190.

\bibitem{Morrey} C.B. Morrey,
On the solutions of quasi-linear elliptic partial differential equations.
{\it Trans. Amer. Math. Soc.} 43(1938)126-166.

\bibitem{Mos} J. Moser, A sharp form of an inequality by N. Trudinger. {\it Indiana Univ. Math. J.} 20(1971)1077-1092.



\bibitem{PT} N.C. Phuc and M. Torres, {Characterizations of the existence and removable singularities of divergence-measure vector fields}. {\it Indiana Univ. Math. J.} 57(2008)1573-1597.



\bibitem{PT1} N.C. Phuc and M. Torres, {Characterizations of signed measures in the dual of $BV$ and related isometric isomorphisms}. {\it Ann. Sc. Norm. Super. Pisa Cl. Sci.} (5) 17(2017)385-417.

\bibitem{PS} A.C. Ponce and D. Spector, A boxing inequality for the fractional perimeter. To appear in
{\it Ann. Scuola Norm. Sup. Pisa Cl. Sci.}

\bibitem{Sa} S.G. Samko, Best constant in the weighted Hardy inequality: the spatial and spherical version. {\it Fract. Calc. Anal. Appl.} 8(2005)39-52.

\bibitem{SSS17} A. Schikorra, D. Spector and  J. Van Schaftingen,
An $L^1$-type estimate for Riesz potentials. \emph{Rev. Mat. Iberoam.} 33(2017)291-303.


\bibitem{SSS} A. Schikorra, T.-T. Shieh and D. Spector, Regularity for a fractional $p$-Laplace equation.  \emph{Commun. Contemp. Math.} 20(2018), no 1, 1760003, 6pp.

\bibitem{SS1} T.-T. Shieh and D. Spector, On a new class of fractional partial differential equations. \emph{Adv. Calc. Var.} 8(2015)321-336.



\bibitem{SS2} T.-T. Shieh and D. Spector, On a new class of fractional partial differential equations II.  \emph{Adv. Calc. Var.} 11(2018)289-307.

\bibitem{Sil} L. Silvestre, Regularity of the obstacle problem for a fractional power of the Laplace operator.
\emph{Comm. Pure Appl. Math.} 60(2007)67-112.

\bibitem{S} {}{M. $\check{\text S}$ilhav\'y}, Fractional vector analysis based on invariance requirements. Preprint No.11-2018, Praha 2018.

\bibitem{Sp} D. Spector, An sharp Sobolev embedding for $L^1$. {\it ArXiv: 1806.07588v2[math.FA] 5 Sep 2019.}

\bibitem{St} E.M. Stein, {\it Singular Integrals and Differentiability Properties of Functions.} Princeton University Press, Princeton, N.J. 1970.

\bibitem{St2} E.M. Stein, \emph{Harmonic Analysis: Real-Variable Methods, Orthogonality, and Oscillatory Integrals.}
Princeton University Press, Princeton, NJ, 1993.

\bibitem{SW} E.M. Stein and G. Weiss, Fractional integrals on $n$-dimensional Euclidean sapce. \emph{J. Math. Mech.} 7(1958)503-514.

\bibitem{Ta} G. Talenti, Inequalities in rearrangement invariant function spaces. {\it Nonlinear Analysis, Function Spaces and Applications}. Vol. 5(1994)177-230.


\bibitem{U}
A. Uchiyama,
A constructive proof of the Fefferman-Stein decomposition of $\mathrm{BMO}(\mathbb R^n)$. \emph{Acta Math.} 148(1982)215-241.


\bibitem{Y} K. Yosida, \emph{Functional Analysis.} Reprint of the sixth (1980) edition. Classics in Mathematics. Springer-Verlag, Berlin, 1995.


\end{thebibliography}
\end{document}